\documentclass[11pt,a4paper]{article}
\usepackage{amsmath,amssymb,latexsym,theorem,graphicx,mathrsfs}
\usepackage{secdot}

\usepackage{xcolor}
\usepackage{url}
\usepackage{hyperref}
\definecolor{ForestGreen}{rgb}{0.1,0.6,0.05}
\definecolor{EgyptBlue}{rgb}{0.063,0.1,0.6}
\definecolor{RipeOlive}{HTML}{556B2F}
\hypersetup{
	colorlinks=true,
	linkcolor=EgyptBlue,         
	citecolor=ForestGreen,
	urlcolor=RipeOlive
}

\usepackage[hyperpageref]{backref}
\renewcommand*{\backref}[1]{}
\renewcommand*{\backrefalt}[4]
{%
	\ifcase #1 (Not cited.)%
	\or        (Cited on page~#2)
	\else      (Cited on pages~#2)
	\fi
}
\usepackage{epstopdf}
\newtheorem{theorem}{Theorem}
\newtheorem{proposition}[theorem]{Proposition}
\newtheorem{lemma}[theorem]{Lemma}

\newtheorem{remark}[theorem]{Remark}

\numberwithin{equation}{section}
\numberwithin{theorem}{section}

\DeclareMathOperator*{\esssup}{ess\,sup}

\newcommand{\qed}{{\unskip\nobreak\hfil%
		\penalty50\hskip .001pt\hbox{}\nobreak\hfil
		\vrule height 1.7ex width .9ex depth .2ex
		\parfillskip=0pt\finalhyphendemerits=0\medbreak}\rm}
\newenvironment{proof}{\begin{trivlist}\item[\hskip%
		\labelsep{{\em Proof.}\ }]\rm}%
	{\hfill\qed\rm\end{trivlist}}

\newenvironment{proof*}[1]{\begin{trivlist}\item[\hskip%
		\labelsep{{\bf Proof of \/{\rm\bf #1.}}\quad}]\rm}%
	{\hfill\qed\rm\end{trivlist}}

\newcommand{\W}{W_0^{1,p}}
\newcommand{\intO}{\int_\Omega}
\newcommand{\C}{C^1_0(\overline{\Omega})}

\newcommand{\E}{E_{\alpha,\beta}}

\usepackage[utf8]{inputenc}
\usepackage{amsfonts}
\usepackage[left=2cm,right=2cm,top=2cm,bottom=2cm]{geometry}
\title{On sign-changing solutions for $(p,q)$-Laplace equations\\ with two parameters
	}
\author{ 
	\normalsize Vladimir Bobkov\\ 
	{\small  Institute of Mathematics, Ufa Scientific Center, Russian Academy of Sciences}\\
	\small{Chernyshevsky str. 112, Ufa 450008, Russia} \\
	{\small  Department of Mathematics and NTIS, Faculty of Applied Sciences, University of West Bohemia}\\ 
	{\small Univerzitn\'i 8, Plze\v{n} 306 14, Czech Republic}\\
	{\small e-mail: bobkov@kma.zcu.cz}\\[0.5em] 
	\normalsize Mieko Tanaka\\
	{\small Department of  Mathematics, 
		Tokyo University of Science}\\
	{\small Kagurazaka 1-3, Shinjyuku-ku, Tokyo 162-8601, Japan}\\
	{\small e-mail: miekotanaka@rs.tus.ac.jp}
}

\date{}

\begin{document}
	\maketitle 
	
\begin{abstract}
	We investigate the existence of nodal (sign-changing) solutions to the Dirichlet problem for a two-parametric family of partially homogeneous $(p,q)$-Laplace equations 
	$-\Delta_p u -\Delta_q u=\alpha |u|^{p-2}u+\beta |u|^{q-2}u$
	where $p \neq q$. By virtue of the Nehari manifolds, the linking theorem, and descending flow, we explicitly characterize subsets of the $(\alpha,\beta)$-plane which correspond to the existence of nodal solutions. In each subset the obtained solutions have prescribed signs of energy and, in some cases, exactly two nodal domains. The nonexistence of nodal solutions is also studied. 
	Additionally, we explore several relations between eigenvalues and eigenfunctions of the $p$- and $q$-Laplacians in one dimension.
	
\end{abstract}

\section{Introduction}

In this article we study the existence and nonexistence of sign-changing solutions for the problem
\begin{equation*}
\tag{$GEV;\alpha,\beta$}
\left\{
\begin{aligned}
-&\Delta_p u  -  \Delta_q u = \alpha |u|^{p-2} u + \beta |u|^{q-2} u 
&&\text{in } \Omega, \\[0.4em]
&u = 0  &&\text{on } \partial \Omega,
\end{aligned}
\right.
\end{equation*}	
where $\Omega \subset \mathbb{R}^N$, $N \geq 1$, is a bounded domain with a sufficiently smooth boundary $\partial \Omega$, and  $\alpha,\beta \in \mathbb{R}$ are parameters.
The operator $\Delta_r u := \text{div}(|\nabla u|^{r-2} \nabla u)$ is the classical $r$-Laplacian, $r = \{q, p\} > 1$, and without loss of generality we assume that $q < p$. 

Boundary value problems with a combination of several differential operators of different nature (in particular, as in $(GEV;\alpha,\beta)$) arise mainly as mathematical models of physical processes and phenomena, and have been extensively studied in the last two decades, see, e.g., \cite{CherIl, quorin, he, perera} and the references below. Among the historically first examples one can mention the Cahn--Hilliard equation \cite{cahnhill} describing the process of separation of binary alloys, and the Zakharov equation \cite[(1.8)]{zakharov} which describes the behavior of plasma oscillations.
Elliptic equations with the $(2,6)$- and $(2,p)$-Laplacians were considered  explicitly in \cite{BenciSoliton,BenciDerrick} with the aim of obtaining soliton-type solutions (in particular, as a model for elementary particles). 

The considered problem $(GEV;\alpha,\beta)$ attracts special attention due to its symmetric and partially homogeneous structure, cf. \cite{T-2014,T-Uniq,MT,KTT,ZK,BobkovTanaka2015,barileeig}. 
By developing the results of \cite{T-2014,MT,KTT}, the authors of the present article obtained in \cite{BobkovTanaka2015} a reasonably complete description of the subsets of the $(\alpha, \beta)$-plane which correspond to the existence/nonexistence of \textit{positive} solutions to the problem  $(GEV;\alpha,\beta)$.
At the same time, to the best of our knowledge, analogous results for \textit{sign-changing} solutions have not been obtained circumstantially so far, although a particular information on the existence can be extracted from \cite{marano2013, T-Uniq, papa2}. 
The main reason for this is a crucial dependence of the structure of the solution set to the problem $(GEV;\alpha,\beta)$ on parameters $\alpha$ and $\beta$. As a consequence, the existence can not be treated by a unique approach, and various tools have to be used for different parts of the $(\alpha, \beta)$-plane.

The aim of the present article is to allocate and characterize the sets of parameters $\alpha$ and $\beta$ for which the problem $(GEV;\alpha,\beta)$ possesses or does not possess sign-changing solutions (see Figure~\ref{fig1}). 
In this sense, this work can be seen as the second part of the article \cite{BobkovTanaka2015}.

\subsection{Notations and preliminaries}\label{subsec:notations} 
Before formulating the main results, we introduce several notations. 
In what follows, $L^r(\Omega)$ with $r \in (1,+\infty)$ and $L^\infty(\Omega)$ stand for the Lebesgue spaces with the norms 
$$
\|u\|_r := \left( \int_\Omega |u|^r \, dx \right)^{1/r} \quad \text{and} \quad 
\|u\|_\infty:= \esssup\limits_{x\in\Omega}|u(x)|,
$$ 
respectively, and  $W_0^{1,r}:=W_0^{1,r}(\Omega)$ denotes the Sobolev space with the norm $\|\nabla u \|_r$. For $u \in W_0^{1,r}$ we define 
$u^\pm := \max\{\pm u,0\}$. Note that $u^\pm \in W_0^{1,r}$ and $u=u^+-u^-$.

By a (weak) solution of $(GEV;\alpha,\beta)$ we mean function $u \in \W$ which satisfies
\begin{equation}\label{weaksolution}
\intO |\nabla u|^{p-2}\nabla u\nabla\varphi \,dx 
+\intO |\nabla u|^{q-2}\nabla u\nabla\varphi \,dx 
= \alpha \intO |u|^{p-2}u \, \varphi \, dx + \beta \intO |u|^{q-2}u \, \varphi\,dx
\end{equation}
for all $\varphi\in \W$. 
If $u$ is a solution of $(GEV;\alpha, \beta)$ and $u^\pm \not\equiv 0$ (a.e.\ in $\Omega$), then $u$ is called \textit{nodal} or \textit{sign-changing} solution. 
It is not hard to see that \textit{any}  solution of $(GEV;\alpha, \beta)$  is a critical point of the energy functional $E_{\alpha,\beta} \in C^1(\W, \mathbb{R})$ defined by
\begin{equation*}
E_{\alpha, \beta}(u) := \frac{1}{p} H_\alpha(u) + \frac{1}{q} G_\beta(u), 
\end{equation*}
where 
\begin{equation*}
H_\alpha(u) := \int_\Omega |\nabla u|^p \, dx - \alpha \int_\Omega |u|^p \,dx 
\quad \text{ and } \quad 
G_\beta(u) := \int_\Omega |\nabla u|^q \, dx - \beta \int_\Omega |u|^q \,dx. 
\end{equation*}
Notice that the supports of $u^+$ and $u^-$ are disjoint for any $u \in \W$. This fact, together with evenness  of the functionals $H_\alpha$ and $G_\beta$, easily implies that 
\begin{equation*}
H_\alpha(u^+) + H_\alpha(u^-) = H_\alpha(u)
\quad \text{ and } \quad 
G_\beta(u^+)+G_\beta(u^-)=G_\beta(u).
\end{equation*}

\begin{remark} 
	Any solution $u \in \W$ of the problem $(GEV;\alpha, \beta)$ belongs to $C^{1,\gamma}_0(\overline{\Omega})$ for some $\gamma \in (0,1)$. 
	In fact, $u\in L^\infty(\Omega)$ by the Moser iteration process, cf. \cite[Appendix~A]{MMT}. Furthermore, the regularity up to the boundary in \cite[Theorem~1]{Lieberman} and \cite[p.~320]{L} provides $u \in
	C^{1,\gamma}_0(\overline{\Omega})$, $\gamma \in (0,1)$.
\end{remark}

Next, we recall several facts related to the eigenvalue problem for the Dirichlet $r$-Laplacian, $r>1$. 
We say that $\lambda$ is an \textit{eigenvalue} of $-\Delta_r$, if the problem
\begin{equation*}
\tag{$EV;r,\lambda$}
\left\{
\begin{aligned}
-&\Delta_r u =\lambda |u|^{r-2}u
&& \text{in } \Omega, \\
&u=0 && \text{on } \partial \Omega 
\end{aligned}
\right.
\end{equation*}
has a nontrivial (weak) solution. 
Analogously to the linear case, the set of all eigenvalues of $(EV;r,\lambda)$ will be denoted as  $\sigma(-\Delta_r)$. 
It is well known that the lowest positive eigenvalue $\lambda_1(r)$ can be obtained through the nonlinear Rayleigh quotient as (cf. \cite{anane1987})
\begin{equation}\label{charact-1st-ev}
\lambda_1(r)
:=\inf\left\{\frac{\intO |\nabla u|^r\,dx}{\intO |u|^r\,dx}:~ u \in W_0^{1,r},~ u \not\equiv 0\right\}.
\end{equation}
The eigenvalue $\lambda_1(r)$ is simple and isolated, and the corresponding eigenfunction $\varphi_r \in \W$ (defined up to an arbitrary multiplier) is strictly positive (or strictly negative) in $\Omega$. 
Moreover, $\lambda_1(r)$ is the unique eigenvalue with a corresponding  sign-constant eigenfunction \cite{anane1987}. Note also that any eigenfunction $\varphi$ of $-\Delta_r$ belongs to $C^{1,\gamma}_0(\overline{\Omega})$ for some $\gamma \in (0,1)$.

The following lemma directly follows from the definition of $\lambda_1(r)$ and its simplicity.
\begin{lemma}\label{lem:1}
Assume that $u \in \W\setminus\{0\}$. Then we have the following results:
\begin{itemize}
\item[{\rm (i)}] 
Let $\alpha \leq \lambda_1(p)$. Then $H_\alpha(u) \geq 0$, and $H_\alpha(u) = 0$ if and only if $\alpha = \lambda_1(p)$ and $u = t\varphi_p$ 
for some $t\in\mathbb{R}\setminus\{0\}$.
\item[{\rm (ii)}] 
Let $\beta \leq \lambda_1(q)$. Then $G_\beta(u) \geq 0$, and $G_\beta(u) = 0$ if and only if $\beta = \lambda_1(q)$ and $u = t\varphi_q$ 
for some $t\in\mathbb{R}\setminus\{0\}$.
\end{itemize}
\end{lemma}

Although the structure of $\sigma(-\Delta_r)$ is not completely known except for the case $r=2$ or $N=1$ (see, e.g., \cite[Theorem~3.1]{drabman}), several unbounded sequences of eigenvalues can be introduced by virtue of minimax variational principles. In what follows, by
$\{\lambda_k(r)\}_{k \in \mathbb{N}}$ we denote a sequence of eigenvalues for $(EV; r, \lambda)$ introduced in \cite{DR}. It can be described variationally as 
\begin{equation}\label{lambda_n}
\lambda_k(r) := \inf_{h\in\mathscr{F}_k(r)} \max_{z\in S^{k-1}} 
\|\nabla h(z)\|_r^r, 
\end{equation}
where $S^{k-1}$ is the unit sphere in $\mathbb{R}^k$ and
\begin{align} 
\label{F_n} 
\mathscr{F}_k(r)
&:=\{ h\in C(S^{k-1},S(r)):~ h \text{ is odd}\},
\\
\notag
S(r) 
&:=\{u\in W_0^{1,r}:~ \|u\|_r =1\}.
\end{align} 
It is known \cite{DR} that $\lambda_k(r) \to +\infty$ as $k \to +\infty$. 
Moreover, $\lambda_2(r)$ coincides with the second eigenvalue of 
$-\Delta_r$, i.e.,
$$
\lambda_2(r) = \inf\{\lambda \in \sigma(-\Delta_r):~ \lambda > \lambda_1(r)\},
$$
and it can be alternatively characterized as in \cite{cuesta}: 
\begin{gather}
\label{second:mp}
\lambda_2(r) =\inf_{\gamma\in\Gamma}\max_{s \in [0,1]} 
\|\nabla \gamma(s)\|_r^r,
\\ 
\notag
\Gamma:=\{\gamma\in C([0,1], S(r)):~ \gamma(0) = \varphi_r,\ 
\gamma(1) = -\varphi_r \},
\end{gather}
where the first eigenfunction  $\varphi_r$ is normalized such that $\varphi_r \in S(r)$. 
We denote any eigenfunction corresponding to $\lambda_2(r)$ as $\varphi_{2,r}$. 
Notice that $\lambda_2(r) > \lambda_1(r)$. 
Furthermore, in the one-dimensional case the sequence \eqref{lambda_n} describes the whole $\sigma(-\Delta_r)$ (cf. \cite[Theorem~4.1]{drabman}, where this result is proved for the Krasnosel'skii-type eigenvalues).

Finally, we introduce the notation for the eigenspace of $-\Delta_r$ at $\lambda \in \mathbb{R}$:
\begin{equation}\label{def:eigenspace}
ES(r;\lambda)
:=\{v\in W_0^{1,r}:~ v \text{ is a  solution of } (EV;r,\lambda)\,\}.
\end{equation}
It is clear that $ES(r;\lambda) \neq \{0\}$ if and only if 
$\lambda \in \sigma(-\Delta_r)$.

\subsection{Main results}\label{subsec:main}
Let us state the main results of this article. 
We begin with the nonexistence of nodal solutions for $(GEV; \alpha, \beta)$. 
\begin{theorem}\label{thm:nonexist} 
Assume that 
$$
(\alpha,\beta)\in (-\infty,\lambda_2(p)]\times(-\infty,\lambda_1(q)]
\cup (-\infty,\lambda_1(p)]\times(-\infty,\lambda_2(q)]. 
$$
Then $(GEV;\alpha,\beta)$ has no nodal solutions. 
\end{theorem}

In the one-dimensional case  Theorem~\ref{thm:nonexist} can be refined as follows.
\begin{theorem}\label{nonexist:N1} 
Let $N=1$. If $(\alpha,\beta)\in (-\infty,\lambda_2(p)]\times 
(-\infty,\lambda_2(q)]$, 
then $(GEV;\alpha,\beta)$ has no nodal solutions. 
\end{theorem}
In the case of general dimensions an additional information on hypothetical nodal solutions to $(GEV; \alpha, \beta)$ for $\alpha \in (\lambda_1(p), \lambda_2(p)]$ and $\beta \in (\lambda_1(q), \lambda_2(q)]$ is given in Lemma~\ref{lem:nonexist} below.

\begin{figure}[!h]
	\begin{center}
	\includegraphics[width=0.7\linewidth]{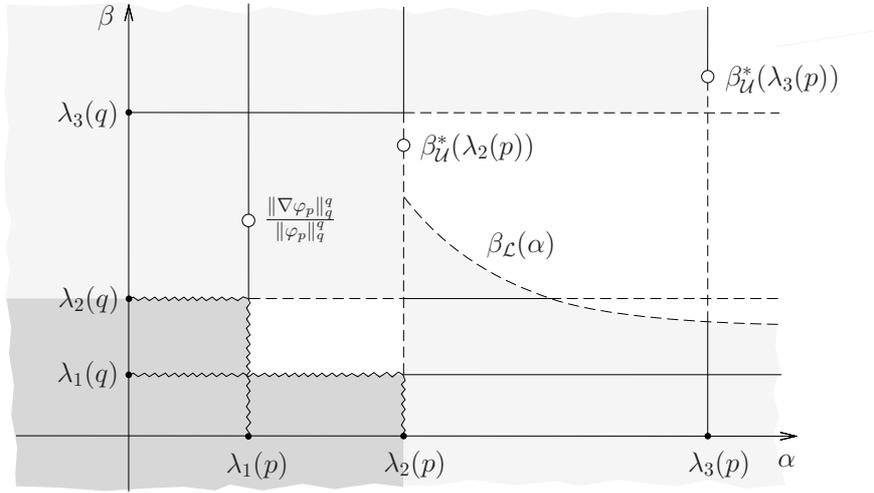}
	\caption{The case $\lambda_2(q) < \lambda_3(q)$, $\lambda_2(p) < \lambda_3(p)$, and $(\lambda_2(p), \lambda_3(p)) \cap \sigma(-\Delta_p) = \emptyset$. Existence (light gray, solid lines), nonexistence (dark gray, zigzag lines), unknown (white, dashed lines)}
	\label{fig1}
	\end{center}
\end{figure}

Now we formulate the existence result for nodal solutions with a positive energy. 
Let us define the following ``lower'' critical value depending on $\alpha \in \mathbb{R}$:
\begin{equation}\label{bK}
\beta_{\mathcal{L}}(\alpha) 
:= 
\inf \left\{\,
\min \left\{ \frac{\intO |\nabla u^+|^q \, dx}{\intO |u^+|^q \, dx}, 
\frac{\intO |\nabla u^-|^q \, dx}{\intO |u^-|^q \, dx} \right\}:~ u \in \mathcal{B}_{\mathcal{L}}(\alpha) \right\},
\end{equation}
where
\begin{equation}\label{BK}
\mathcal{B}_\mathcal{L}(\alpha) := 
\left\{u \in \W:~ u^\pm \not\equiv 0,~ \max 
\left\{\frac{\intO |\nabla u^+|^p \, dx}{\intO |u^+|^p \, dx}, \frac{\intO |\nabla u^-|^p \, dx}{\intO |u^-|^p \, dx} 
\right\} 
\leq \alpha
\right\},
\end{equation}
and put $\beta_\mathcal{L}(\alpha) = +\infty$ whenever the admissible set 
$\mathcal{B}_{\mathcal{L}}(\alpha)$ is empty. 
\begin{theorem}\label{thm:positive} 
	Let $\alpha > \lambda_2(p)$. Then for all
	$\beta < \beta_{\mathcal{L}}(\alpha)$ the problem $(GEV;\alpha,\beta)$ has a nodal solution $u$ with $\E(u)>0$ and  precisely two nodal domains. 
\end{theorem}

Several main properties of the function $\beta_\mathcal{L}(\alpha)$ are collected in Lemma~\ref{lem:bK} below. Let us remark that the parametrization by $\alpha$ in \eqref{bK} is different from the parametrization by $s$ of the form $(\alpha, \beta) = (\lambda+s, \lambda)$ which was used in \cite{BobkovTanaka2015} in order to construct a critical curve for the existence of positive solutions. 
In the context of the present article, the parametrization by $\alpha$ makes the problem $(GEV; \alpha, \beta)$ easier to analyze. 
We also note that \eqref{bK} is conceptually similar to the characterization of the first nontrivial curve of the Fu\v{c}\'ik spectrum given in \cite{mollepas}. In Section~\ref{sec:nehari} below, we introduce and study several other critical points besides \eqref{bK}, which although are not directly used in the proofs of the main results, increase the understanding of the construction of $(\alpha, \beta)$-plane, and could be employed in further investigations.

Next, we state the existence of negative energy nodal solutions for $(GEV;\alpha, \beta)$. 
Consider the ``upper'' critical value 
\begin{equation}\label{b_*2}
\beta^{*}_{\mathcal{U}}(\alpha) 
:= 
\sup \left\{\,
\frac{\intO |\nabla \varphi|^q \, dx}{\intO |\varphi|^q \, dx}:~ 
\varphi \in ES(p; \alpha)\setminus \{0\} \right\},
\end{equation}
where $\alpha \in \mathbb{R}$, and set $\beta^{*}_{\mathcal{U}}(\alpha) = -\infty$ provided $\alpha \not\in \sigma(-\Delta_p)$.
Several lower and upper bounds for $\beta^*_\mathcal{U}(\alpha)$ are given in Lemmas~\ref{lem:bU:bound1} and \ref{lem:bU:bound2} below.
Define $k_\alpha := \min\{k \in \mathbb{N}:~ \alpha < \lambda_{k+1}(p)\}$ and notice that $\lambda_{k_\alpha+1}(q) \geq \lambda_2(q)$ for all $\alpha \in \mathbb{R}$.
\begin{theorem}\label{thm:neg1}
	Let $\alpha \in \overline{\mathbb{R}\setminus\sigma(-\Delta_p)}$. 
	Then for all $\beta > \max\left\{\beta^{*}_{\mathcal{U}}(\alpha), \lambda_{k_\alpha+1}(q)\right\}$ the problem $(GEV;\alpha,\beta)$ has a nodal solution $u$ satisfying $\E(u) < 0$. 
\end{theorem}

Evidently, if $\sigma(-\Delta_p)$ is a discreet set (as it is for $p=2$ or $N = 1$), then $\overline{\mathbb{R}\setminus\sigma(-\Delta_p)} = \mathbb{R}$. Moreover,  $\lambda_1(p)$ and $\lambda_2(p)$ belong to $\overline{\mathbb{R}\setminus\sigma(-\Delta_p)}$ for all $p>1$ and $N \geq 1$, since $\lambda_1(p)$ is isolated and there are no eigenvalues between $\lambda_1(p)$ and $\lambda_2(p)$ (see Subsection~\ref{subsec:notations}).

One of the main ingredients for the proof of Theorem~\ref{thm:neg1} is the  result on the existence of three nontrivial solutions (positive, negative and sign-changing) to the problem with the $(p,q)$-Laplacian and a nonlinearity in the general form given by Theorem~\ref{thm:1} below. 
This result is of independent interest. 

Theorem~\ref{thm:neg1} can be refined as follows.
\begin{theorem}\label{thm:neg2}
Assume that
$$
(\alpha,\beta) \in (-\infty,\lambda_2(p))\times (\lambda_2(q),+\infty) 
\setminus
\{
(\lambda_1(p),\|\nabla \varphi_p\|_q^q/\|\varphi_p\|_q^q)
\}.
$$
Then $(GEV;\alpha,\beta)$ has a nodal solution $u$ 
satisfying $\E(u)<0$. 
\end{theorem}

\begin{remark} 
In the one-dimensional case we have 
$\|\varphi_p'\|_q^q/\|\varphi_p\|_q^q<\lambda_2(q)$ (see Lemma~\ref{lem:appendixA} in Appendix~A), and hence the assertion of Theorem~\ref{thm:neg2} holds for all 
$(\alpha,\beta) \in (-\infty,\lambda_2(p))\times (\lambda_2(q),+\infty)$. 
\end{remark}

Let us note that unlike the case of positive solutions, the structure of the set of nodal solutions for the problem $(GEV;\alpha, \beta)$ is more complicated, and we are not aware of the maximality of the regions obtained in Theorems~\ref{thm:positive} and \ref{thm:neg1}. 

The article is organized as follows. In Section~\ref{sec:nehari}, we apply the method of the Nehari manifold in order to prove Theorem~\ref{thm:positive}. 
In Section~\ref{sec:minimax}, by means of linking arguments and the descending flow method, we provide two general existence results which yield, in particular, Theorems~\ref{thm:neg1} and \ref{thm:neg2}. 
For the convenience of the reader we collect the proofs of the main theorems in Section~\ref{sec:proofs}. 
In Appendix~A, we prove several additional facts on the relation between eigenvalues and eigenfunctions of the $p$- and $q$-Laplacians in the one-dimensional case. Finally, in Appendix~B, we give a sketch of the proof of Theorem~\ref{thm:1}.

\section{Nodal solutions with positive energy}\label{sec:nehari}
The classical Nehari manifold for the problem
$(GEV;\alpha,\beta)$ is defined by 
$$
\mathcal{N}_{\alpha, \beta} := 
\left\{ 
u \in \W \setminus \{0\}:~
\left<E_{\alpha, \beta}'(u), u \right> = H_\alpha(u) + G_\beta(u) = 0
\right\}.
$$
It can be readily seen that $\mathcal{N}_{\alpha, \beta}$ 
contains \textit{all} nontrivial solutions of $(GEV;\alpha,\beta)$. 
On the other hand, if $u \in W_0^{1,p}$ is a sign-changing solution 
of $(GEV;\alpha,\beta)$, then
\begin{align*}
0=\left< E_{\alpha,\beta}'(u), u^+ \right> 
&= 
\left< E_{\alpha,\beta}'(u^+), u^+ \right> =  H_\alpha(u^+) + G_\beta(u^+), \\
0=-\left< E_{\alpha,\beta}'(u), u^- \right> 
&=
\left< E_{\alpha,\beta}'(u^-), u^- \right> = 
H_\alpha(u^-) + G_\beta(u^-).
\end{align*}
These equalities bring us to the definition of the so-called \textit{nodal} Nehari set for $(GEV;\alpha,\beta)$:
\begin{equation}\label{def:M}
\mathcal{M}_{\alpha, \beta} := \left\{u \in \W:~ u^\pm \not\equiv 0,~  H_\alpha(u^\pm) + G_\beta(u^\pm) = 0 \right\} 
= 
\left\{u \in W_0^{1,p}:~ u^\pm \in \mathcal{N}_{\alpha, \beta} \right\}.
\end{equation}
By construction, $\mathcal{M}_{\alpha, \beta}$ contains \textit{all} sign-changing solutions of $(GEV;\alpha,\beta)$, and hence $\mathcal{M}_{\alpha, \beta} \subset \mathcal{N}_{\alpha, \beta}$. 

Let us divide $\mathcal{M}_{\alpha, \beta}$ into the following three subsets: 
\begin{align*}
\mathcal{M}_{\alpha, \beta}^1 &:= \left\{u \in \mathcal{M}_{\alpha, \beta}:~ H_\alpha(u^+) < 0,~ H_\alpha(u^-)<0 \right\}, \\
\mathcal{M}_{\alpha, \beta}^2 &:= \left\{u \in \mathcal{M}_{\alpha, \beta}:~ H_\alpha(u^+) > 0,~ H_\alpha(u^-)>0 \right\}, \\
\mathcal{M}_{\alpha, \beta}^3 &:= \left\{u \in \mathcal{M}_{\alpha, \beta}:~ H_\alpha(u^+) \cdot H_\alpha(u^-) \leq 0 \right\}. 
\end{align*}
Evidently, $\mathcal{M}_{\alpha, \beta} = \mathcal{M}_{\alpha, \beta}^1 \cup \mathcal{M}_{\alpha, \beta}^2 \cup \mathcal{M}_{\alpha, \beta}^3$ 
and all $\mathcal{M}_{\alpha, \beta}^i$ are mutually disjoint. 
The main aim of this section is to prove the existence of nodal solutions for $(GEV;\alpha,\beta)$ through minimization of $E_{\alpha,\beta}$ over $\mathcal{M}_{\alpha, \beta}^1$ in an appropriate subset of the $(\alpha,\beta)$-plane.

\subsection{Preliminary analysis}
In this subsection, we mainly study the properties of the sets 
$\mathcal{M}_{\alpha, \beta}^1$, $\mathcal{M}_{\alpha, \beta}^2$, and $\mathcal{M}_{\alpha, \beta}^3$. First of all, we give the following auxiliary lemma, which is in fact analogous to \cite[Proposition~6]{BobkovTanaka2015} and can be proved in the same manner.
\begin{lemma}
	\label{lemm:E>0}
	Let $u \in W_0^{1,p}$. If $H_\alpha(u) \cdot G_\beta(u) < 0$, then there exists a unique critical point $t(u) > 0$ of $E_{\alpha, \beta}(t u)$ with respect to $t > 0$, and $t(u)u \in \mathcal{N}_{\alpha, \beta}$.
	In particular, if
	\begin{equation*}
	H_\alpha(u)< 0 <G_\beta(u),
	\end{equation*}
	then $t(u)$ is the unique maximum point of  $E_{\alpha, \beta}(t u)$ with respect to $t > 0$, and $E_{\alpha, \beta}(t(u) u) > 0$.
\end{lemma}

We start our consideration of the sets $\mathcal{M}_{\alpha, \beta}^i$ with several simple facts. 
\begin{lemma}\label{lem:easy-fact} 
	Let $\alpha, \beta \in \mathbb{R}$. The following hold:
\begin{itemize} 
\item[{\rm (i)}] If $\beta\leq \lambda_1(q)$, then 
$\mathcal{M}_{\alpha, \beta}^1 = \mathcal{M}_{\alpha, \beta}$ and, consequently,  $\mathcal{M}_{\alpha, \beta}^2$,  $\mathcal{M}_{\alpha, \beta}^3=\emptyset$.
\item[{\rm (ii)}] If $\alpha\leq \lambda_1(p)$, then 
$\mathcal{M}_{\alpha, \beta}^2 = \mathcal{M}_{\alpha, \beta}$ and, consequently,  $\mathcal{M}_{\alpha, \beta}^1$, $\mathcal{M}_{\alpha, \beta}^3=\emptyset$.
\end{itemize}
\end{lemma}
\begin{proof} 
Let us first prove the assertion (i). Assume that $\beta\leq \lambda_1(q)$ and $w\in \mathcal{M}_{\alpha, \beta}$. 
Then Lemma~\ref{lem:1} implies that $G_\beta(w^\pm) \geq 0$ and in fact $G_\beta(w^\pm) > 0$, since otherwise $w^\pm = \varphi_q$, which is impossible in view of the strict positivity of $\varphi_q$ in $\Omega$.
Thus, the Nehari constraints $H_\alpha(w^\pm) + G_\beta(w^\pm) = 0$ yield $H_\alpha(w^\pm) < 0$,  whence $w \in \mathcal{M}_{\alpha, \beta}^1$. 
The assertion (ii) can be shown by the same arguments. 
\end{proof}

Let us introduce the following sets:
\begin{align}
\label{eq:B1}
\mathcal{B}_1(\alpha) 
&:= 
\left\{u \in \W:~  H_\alpha(u^+)<0,~ H_\alpha(u^-)<0\,\right\}, \\
\label{eq:B2}
\mathcal{B}_2(\alpha) 
&:= 
\left\{u \in \W:~  H_\alpha(u^+)>0,~ H_\alpha(u^-)>0\,\right\}.
\end{align}
Obviously, $\mathcal{M}_{\alpha, \beta}^1 \subset \mathcal{B}_1(\alpha)$ and $\mathcal{M}_{\alpha, \beta}^2 \subset \mathcal{B}_2(\alpha)$. Moreover, we have the following result.
\begin{lemma}\label{lem:easy-fact2}
	Let $\alpha, \beta \in \mathbb{R}$. The following hold:
	\begin{itemize} 
		\item[{\rm (i)}] If $\alpha \leq \lambda_2(p)$, then $\mathcal{B}_1(\alpha) = \emptyset$ and, consequently, 
		$\mathcal{M}_{\alpha, \beta}^1 = \emptyset$.
		\item[{\rm (ii)}] If $\beta  \leq \lambda_2(q)$, then $\mathcal{B}_2(\alpha) = \emptyset$ and, consequently, 
		$\mathcal{M}_{\alpha, \beta}^2 = \emptyset$.
	\end{itemize}
\end{lemma}	
\begin{proof}
	We give the proof of the assertion (i). The second part can be proved analogously. Suppose, by contradiction, that $\alpha \leq \lambda_2(p)$ and there exists $w \in \mathcal{B}_1(\alpha)$. These assumptions read as 
	\begin{equation}\label{eq:2nd-ev-1} 
	\max \left\{ \frac{\intO |\nabla w^+|^p \, dx}{\intO |w^+|^p \, dx}, 
	\frac{\intO |\nabla w^-|^p \, dx}{\intO |w^-|^p \, dx} \right\}< \alpha \le \lambda_2(p). 
	\end{equation}
	On the other hand, it is shown in \cite[Proposition~4.2]{BobkovEJQT2014} 
	that the second eigenvalue $\lambda_2(r)$, $r > 1$, can be characterized as follows: 
	\begin{equation}\label{eq:char-2nd-ev}
	\lambda_2(r)=\inf\left\{
	\max \left\{\frac{\intO |\nabla u^+|^r \, dx}{\intO |u^+|^r \, dx}, 
	\frac{\intO |\nabla u^-|^r \, dx}{\intO |u^-|^r \, dx} \right\}:~ u\in W^{1,r}_{0},\ 
	u^\pm \not\equiv 0 
	\right\}.
	\end{equation}
	Comparing \eqref{eq:2nd-ev-1} and \eqref{eq:char-2nd-ev} (with $r=p$), we obtain a contradiction. 
\end{proof}

Lemmas~\ref{lem:easy-fact} and \ref{lem:easy-fact2} readily entail the following information about the emptiness of $\mathcal{M}_{\alpha, \beta}$ and, consequently, the nonexistence of nodal solutions for $(GEV;\alpha,\beta)$. 
\begin{lemma}\label{lemma:emptinessM}
If $\alpha \leq \lambda_2(p)$ and $\beta \leq \lambda_1(q)$, or 
$\alpha \leq \lambda_1(p)$ and $\beta \leq \lambda_2(q)$, 
then $\mathcal{M}_{\alpha, \beta} = \emptyset$. 
\end{lemma}

\begin{lemma}\label{lem:nonexist} 
Let $\alpha \leq \lambda_2(p)$ and $\beta \leq \lambda_2(q)$. 
If $u$ is a nodal solution of $(GEV;\alpha,\beta)$, then 
$\alpha > \lambda_1(p)$, $\beta > \lambda_1(q)$ and
$u\in\mathcal{M}_{\alpha,\beta}^3$. 
\end{lemma}

Let us now subsequently treat the emptiness and nonemptiness of $\mathcal{M}_{\alpha, \beta}^1$ and $\mathcal{M}_{\alpha, \beta}^2$. 
First we consider $\mathcal{M}_{\alpha,\beta}^1$. 
Introduce the critical value 
\begin{align*}
\beta_1(\alpha) := 
\sup \left\{
\min \left\{ \frac{\intO |\nabla u^+|^q \, dx}{\intO |u^+|^q \, dx}, 
\frac{\intO |\nabla u^-|^q \, dx}{\intO |u^-|^q \, dx} \right\}
:~ u\in \mathcal{B}_1(\alpha)\, 
\right\}
\end{align*}
for each $\alpha \in \mathbb{R}$,
where the admissible set $\mathcal{B}_1(\alpha)$ is defined by \eqref{eq:B1}, or, equivalently,
\begin{align*}
\mathcal{B}_1(\alpha) = 
\left\{u \in \W:~ u^\pm \not\equiv 0,~ \max 
\left\{\frac{\intO |\nabla u^+|^p \, dx}{\intO |u^+|^p \, dx}, \frac{\intO |\nabla u^-|^p \, dx}{\intO |u^-|^p \, dx} 
\right\} 
< \alpha
\right\}.
\end{align*}
We assume that $\beta_1(\alpha) = -\infty$ whenever $\mathcal{B}_1(\alpha)$ is empty. Consider also  
\begin{equation}\label{def:beta_1}
\beta_1^*:=\sup\left\{
\min
\left\{ \frac{\intO |\nabla \varphi^+|^q \, dx}{\intO |\varphi^+|^q \, dx}, 
\frac{\intO |\nabla \varphi^-|^q \, dx}{\intO |\varphi^-|^q \, dx} \right\}:~ 
\varphi\in ES(p,\lambda_2(p))\setminus\{0\}\right\}, 
\end{equation}
where $ES(p,\lambda_2(p))$ is the eigenspace of the second eigenvalue $\lambda_2(p)$ defined by \eqref{def:eigenspace}. 

The main properties of $\beta_1(\alpha)$ are collected in the following lemma.
\begin{lemma}\label{lem:b1}
	The following assertions hold:
\begin{enumerate}
	\item[{\rm (i)}] $\beta_1(\alpha) = -\infty$ for any $\alpha \leq \lambda_2(p)$, and $\beta_1(\alpha) \in [\beta_1^*, +\infty)$ for all $\alpha > \lambda_2(p)$; 
	\item[{\rm (ii)}] $\beta_1(\alpha)$ is nondecreasing for $\alpha \in (\lambda_2(p), +\infty)$;
	\item[{\rm (iii)}] $\beta_1(\alpha)$ is left-continuous for $\alpha \in (\lambda_2(p), +\infty)$;
	\item[{\rm (iv)}] $\beta_1(\alpha) \to +\infty$ as $\alpha \to +\infty$; 
	\item[{\rm (v)}] $\mathcal{M}_{\alpha, \beta}^1 \neq \emptyset$ if and only if $\alpha > \lambda_2(p)$ and $\beta < \beta_1(\alpha)$.
\end{enumerate}
\end{lemma}
\begin{proof}
(i) If $\alpha \leq \lambda_2(p)$, then $\mathcal{B}_1(\alpha) = \emptyset$ in view of Lemma~\ref{lem:easy-fact2}, and hence $\beta_1(\alpha) = -\infty$. 
On the other hand, if $\alpha > \lambda_2(p)$, then any second eigenfunction $\varphi_{2,p}$ satisfies $H_\alpha(\varphi_{2,p}^\pm) < 0$ and, in consequence, it belongs to $\mathcal{B}_1(\alpha)$. This implies that  $ES(p,\lambda_2(p))\setminus\{0\}\subset \mathcal{B}_1(\alpha)$ and $\beta_1(\alpha) \geq \beta_1^*$.

Consider the set
\begin{equation}\label{def:X}
X(\alpha) := \{
v \in W_0^{1,p}:~ \|\nabla v\|_p^p \leq \alpha \|v\|_p^p \, \}.
\end{equation}
It is known that for any $\alpha \in \mathbb{R}$ there exists $C(\alpha) > 0$ such that $\|\nabla v\|_p \leq C(\alpha) \|v\|_q$ for all $v \in X(\alpha)$, see \cite[Lemma~9]{T-2014}.
Therefore, since $u^\pm \in X(\alpha)$ for any $u \in \mathcal{B}_1(\alpha)$,  the H\"older inequality yields the existence of a constant $C_1 > 0$ 
such that 
$$
C_1 \|\nabla u^\pm\|_q \leq \|\nabla u^\pm\|_p \leq C(\alpha) \|u^\pm\|_q
~\text{ for all }~
u \in \mathcal{B}_1(\alpha), 
$$
which gives the boundedness of $\beta_1(\alpha)$ from above. 

(ii) If $\lambda_2(p) < \alpha_1 \leq \alpha_2$, then $\mathcal{B}_1(\alpha_1) \subset \mathcal{B}_1(\alpha_2)$, which implies the desired monotonicity.

(iii) Let us fix an arbitrary  $\alpha_0>\lambda_2(p)$. 
Since the assertion (ii) readily leads to $\lim\limits_{\alpha \to \alpha_0-0}\beta_1(\alpha) \leq \beta_1(\alpha_0)$, 
it is enough to show that  $\lim\limits_{\alpha \to \alpha_0-0}\beta_1(\alpha) \geq \beta_1(\alpha_0)$.
By the definition of $\beta_1(\alpha_0)$, for any $\varepsilon > 0$  
there exists $u_\varepsilon \in \mathcal{B}_1(\alpha_0)$ such that 
\begin{equation}\label{eq:b1-1}
\beta_1(\alpha_0)-\varepsilon \leq \min \left\{ \frac{\intO |\nabla u_\varepsilon^+|^q \, dx}{\intO |u_\varepsilon^+|^q \, dx}, 
\frac{\intO |\nabla u_\varepsilon^-|^q \, dx}{\intO |u_\varepsilon^-|^q \, dx} \right\}. 
\end{equation}
Recalling that  $H_{\alpha_0}(u_\varepsilon^\pm)<0$, we can find $\delta = \delta(\varepsilon)>0$ such that $H_\alpha(u_\varepsilon^\pm) < 0$ for any $\alpha \in (\alpha_0-\delta, \alpha_0]$. Therefore, $u_\varepsilon \in \mathcal{B}_1(\alpha)$, and for all $\alpha \in (\alpha_0-\delta, \alpha_0]$ 
the definition of $\beta_1(\alpha)$ leads to 
\begin{equation}\label{eq:b1-2}
\min \left\{ \frac{\intO |\nabla u_\varepsilon^+|^q \, dx}{\intO |u_\varepsilon^+|^q \, dx}, 
\frac{\intO |\nabla u_\varepsilon^-|^q \, dx}{\intO |u_\varepsilon^-|^q \, dx} \right\}
\le \beta_1(\alpha).
\end{equation}
Combining \eqref{eq:b1-1} and \eqref{eq:b1-2}, we obtain the inequality 
$\lim\limits_{\alpha \to \alpha_0-0}\beta_1(\alpha) \geq \beta_1(\alpha_0)$, since $\varepsilon>0$ is arbitrary. 

(iv) Let $L > \lambda_1(q)$ be an arbitrary positive constant. Recalling that for the variational eigenvalues $\lambda_k(q)$ there holds $\lambda_k(q) \to +\infty$ as $k \to +\infty$, we can find $k_L \geq 2$ such that 
$\lambda_{k_L}(q) > L$. Take an eigenfunction $\varphi$ corresponding to $\lambda_{k_L}(q)$. Since $\varphi\in C^{1,\gamma}_0(\overline{\Omega})$ and $\varphi$ changes its sign in $\Omega$ (see Subsection~\ref{subsec:notations}), there exists $\alpha_L$ satisfying 
$$
\max\left\{ \frac{\intO |\nabla \varphi^+|^p \, dx}
{\intO |\varphi^+|^p \, dx}, 
\frac{\intO |\nabla \varphi^-|^p \, dx}{\intO |\varphi^-|^p \, dx}
\right\}<\alpha_L. 
$$
Therefore $\varphi\in \mathcal{B}_1(\alpha_L)$, and 
from the definition of $\beta_1(\alpha_L)$ and 
its monotonicity it follows that 
\begin{equation*}
\beta_1(\alpha) \geq
\beta_1(\alpha_L) \geq \min\left\{ \frac{\intO |\nabla \varphi^+|^q \, dx}
{\intO |\varphi^+|^q \, dx}, 
\frac{\intO |\nabla \varphi^-|^q \, dx}{\intO |\varphi^-|^q \, dx}
\right\} =\lambda_{k_L}(q) > L 
\end{equation*}
provided $\alpha \geq \alpha_L$. 
Since $L$ can be chosen arbitrary large, we conclude that $\lim\limits_{\alpha \to +\infty}\beta_1(\alpha) = +\infty$. 

(v) If $\alpha > \lambda_2(p)$ and $\beta < \beta_1(\alpha)$, then,  
by the definition of $\beta_1(\alpha)$, 
there exists $u \in \mathcal{B}_1(\alpha)$ such that
\begin{equation}\label{eq:b<frac<b1}
\beta < \min \left\{ \frac{\intO |\nabla u^+|^q \, dx}{\intO |u^+|^q \, dx}, \frac{\intO |\nabla u^-|^q \, dx}{\intO |u^-|^q \, dx} \right\} \leq \beta_1(\alpha).
\end{equation}
This means that $H_\alpha(u^\pm) < 0$ and $G_\beta(u^\pm) > 0$. 
Hence, by Lemma~\ref{lemm:E>0} we obtain $t^\pm > 0$ such that 
$t^\pm u^\pm\in\mathcal{N}_{\alpha,\beta}$, whence $t^+ u^+ - t^- u^- \in \mathcal{M}_{\alpha,\beta}^1$. 

Suppose now that there exists $u \in \mathcal{M}_{\alpha,\beta}^1$ for some $\alpha, \beta \in \mathbb{R}$. 
Lemma~\ref{lem:easy-fact2} implies that  $\alpha > \lambda_2(p)$. On the other hand, $u \in \mathcal{M}_{\alpha,\beta}^1 \subset \mathcal{B}_1(\alpha)$. Hence, from the Nehari constraints it follows that $G_\beta(u^\pm) > 0$, and we arrive to \eqref{eq:b<frac<b1}. 
\end{proof}

\smallskip
Consider now the set $\mathcal{M}_{\alpha,\beta}^2$. 
The corresponding critical value, parametrized again by $\alpha \in \mathbb{R}$, appears to be the following:
\begin{align*}
\beta_2(\alpha) 
:= 
\inf\left\{
\max \left\{ \frac{\intO |\nabla u^+|^q \, dx}{\intO |u^+|^q \, dx}, 
\frac{\intO |\nabla u^-|^q \, dx}{\intO |u^-|^q \, dx} \right\}:~ 
u\in \mathcal{B}_2(\alpha)\right\},
\end{align*}
where the admissible set $\mathcal{B}_2(\alpha)$ is defined by \eqref{eq:B2}. 

The main properties of $\beta_2(\alpha)$ are similar to those for $\beta_1(\alpha)$ and collected in the following lemma.
\begin{lemma}\label{lem:b2}
	The following assertions hold:
	\begin{enumerate}
		\item[{\rm (i)}]  $\beta_2(\alpha) \in [\lambda_2(q), +\infty)$ for any $\alpha \in \mathbb{R}$;
		\item[{\rm (ii)}] $\beta_2(\alpha)$ is nondecreasing for $\alpha \in \mathbb{R}$, and $\beta_2(\alpha)=\beta_2(\lambda_1(p))=\lambda_2(q)$ for $\alpha\leq \lambda_1(p)$;
		\item[{\rm (iii)}]
		$\beta_2(\alpha)$ is  right-continuous for $\alpha \in \mathbb{R}$;
		\item[{\rm (iv)}] $\mathcal{M}_{\alpha, \beta}^2 \neq \emptyset$ if and only if $\alpha \in \mathbb{R}$ and $\beta > \beta_2(\alpha)$.
	\end{enumerate}
\end{lemma}
\begin{proof}
(i) It is easy to see that for any $\alpha \in \mathbb{R}$ the admissible set $\mathcal{B}_2(\alpha)$ is nonempty. For example, any eigenfunction corresponding to $\lambda \in \sigma(-\Delta_p)$ belongs to $\mathcal{B}_2(\alpha)$ provided $\lambda> \max\{\alpha,\lambda_1(p)\}$. 
Hence, $\beta_2(\alpha) < +\infty$. 
On the other hand, the definition of $\beta_2(\alpha)$ and characterization  \eqref{eq:char-2nd-ev} with $r=q$ directly imply that $\beta_2(\alpha) \geq \lambda_2(q)$ for any $\alpha \in \mathbb{R}$, since $\mathcal{B}_2(\alpha) \subset W_0^{1,p}\subset W_0^{1,q}$.

(ii) If $\alpha_1 \le \alpha_2$, then $\mathcal{B}_2(\alpha_2) \subset \mathcal{B}_2(\alpha_1)$, which leads to the desired monotonicity. 
Since \textit{any} sign-changing function $w \in \W$ satisfies $H_{\lambda_1(p)}(w^\pm) > 0$ (see Lemma~\ref{lem:1}), we get
$\mathcal{B}_2(\alpha)= 
\mathcal{B}_2(\lambda_1(p))=\{u\in W^{1,p}_0:~ u^\pm \not\equiv 0\}$ 
for all $\alpha\leq \lambda_1(p)$, and hence 
$\beta_2(\alpha)=\beta_2(\lambda_1(p))$ for all $\alpha\leq \lambda_1(p)$. 
In order to show that  $\beta_2(\lambda_1(p))=\lambda_2(q)$, 
let us recall that any eigenfunction $\varphi_{2,q}$ corresponding to $\lambda_2(q)$ belongs to $C_0^{1, \gamma}(\overline{\Omega})$ (see Subsection~\ref{subsec:notations}). Hence,  $\varphi_{2,q}\in \mathcal{B}_2(\lambda_1(p))$ and, consequently, 
$$
\lambda_2(q)
=\max \left\{ \frac{\intO |\nabla \varphi_{2,q}^+|^q \, dx}{\intO |\varphi_{2,q}^+|^q \, dx}, \frac{\intO |\nabla \varphi_{2,q}^-|^q \, dx}{\intO |\varphi_{2,q}^-|^q \, dx} \right\}
\ge \beta_2(\lambda_1(p)) \geq \lambda_2(q), 
$$
where the equality follows from 
\eqref{eq:char-2nd-ev} with $r=q$, and the last inequality is given by the assertion (i).

The assertions (iii) and (iv) can be proved in much the same way as in Lemma~\ref{lem:b1}.
\end{proof}

For the further proof of the existence of nodal solutions to $(GEV;\alpha,\beta)$ in $\mathcal{M}_{\alpha,\beta}^1$, 
let us study the properties of the
 critical value \eqref{bK} defined as
\begin{equation*}
\beta_{\mathcal{L}}(\alpha) := 
\inf \left\{\,
\min \left\{ \frac{\intO |\nabla u^+|^q \, dx}{\intO |u^+|^q \, dx}, 
\frac{\intO |\nabla u^-|^q \, dx}{\intO |u^-|^q \, dx} \right\} 
\,:\,u\in \mathcal{B}_{\mathcal{L}}(\alpha) \right\},
\end{equation*}
where the admissible set $\mathcal{B}_{\mathcal{L}}(\alpha)$ is given by \eqref{BK}, or, equivalently,
\begin{equation*}
\mathcal{B}_{\mathcal{L}}(\alpha) =
\left\{
u \in W_0^{1,p}:~
u^\pm \not\equiv 0,~
H_\alpha(u^+) \leq 0,~ 
H_\alpha(u^-) \leq 0\,
\right\}.
\end{equation*}
We put $\beta_\mathcal{L}(\alpha) = +\infty$ whenever $\mathcal{B}_{\mathcal{L}}(\alpha) = \emptyset$. 
Arguing as in the proof of Lemma~\ref{lem:easy-fact2}, it can be shown that $\mathcal{B}_{\mathcal{L}}(\alpha) = \emptyset$ if and only if $\alpha < \lambda_2(p)$. 
Note that $\mathcal{M}_{\alpha,\beta}^1 \subset \mathcal{B}_{1}(\alpha) \subset \mathcal{B}_{\mathcal{L}}(\alpha)$.

First we give two auxiliary results.
\begin{lemma}\label{lem:normalized1} 
	Let $\alpha>0$, $\beta \in \mathbb{R}$, and $\{u_n\}_{n \in \mathbb{N}}$ be an arbitrary sequence in $\mathcal{B}_{\mathcal{L}}(\alpha)$ (or in $\mathcal{M}_{\alpha,\beta}^1$). Denote by $\{v_n\}_{n \in \mathbb{N}}$ a sequence normalized as follows:
	\begin{equation}\label{normlizedseq}
	v_n:=\frac{u_n^+}{\|\nabla u_n^+\|_p}-\frac{u_n^-}{\|\nabla u_n^-\|_p}, \quad n \in \mathbb{N}.
	\end{equation}
	Then the following assertions hold:
	\begin{itemize}
		\item[{\rm (i)}] $v_n \in \mathcal{B}_{\mathcal{L}}(\alpha)$ (or $v_n \in \mathcal{M}_{\alpha,\beta}^1$) for all $n \in \mathbb{N}$;
		\item[{\rm (ii)}] $v_n$ converges, up to a subsequence, to some $v_0 \in \W$ weakly in $\W$ and strongly in $L^p(\Omega)$;
		\item[{\rm (iii)}] $v_0^\pm \not\equiv 0$ and $H_\alpha(v_0^\pm) \leq 0$, that is, $v_0 \in \mathcal{B}_{\mathcal{L}}(\alpha)$. 
	\end{itemize}
\end{lemma} 
\begin{proof} 
	Obviously, $v_n^\pm = u_n^\pm/\|\nabla u_n^\pm\|_p$, and hence the assertion (i) follows from the $p$-homogeneity of $H_\alpha$.
	The assertion (ii) is a consequence of the boundedness of $\{v_n\}_{n \in \mathbb{N}}$ in $\W$. 
	Since $H_\alpha (v_n^\pm) \leq 0$ for all $n \in \mathbb{N}$, we get $\|v_n^\pm\|_p^p \geq 1/\alpha$, whence $v_0^\pm \not\equiv 0$ a.e.\ in $\Omega$, due to the strong convergence of $v_n$ in $L^p(\Omega)$.
	Moreover, using the weak lower semicontinuity of the $\W$-norm, 
	we conclude that $H_\alpha(v_0^\pm) \leq \liminf\limits_{n \to +\infty} 
	H_\alpha(v_n^\pm) \leq 0$. This is the assertion (iii). 
\end{proof}

\begin{proposition}\label{prop:exist-minimizer}
	For any $\alpha \geq \lambda_2(p)$ there exists a minimizer 
$u_\alpha \in \mathcal{B}_{\mathcal{L}}(\alpha)$ of 
$\beta_\mathcal{L}(\alpha)$. 
\end{proposition}
\begin{proof}
	If $\alpha \geq \lambda_2(p)$, then $\mathcal{B}_{\mathcal{L}}(\alpha)$ is nonempty, since $H_\alpha(\varphi_{2,p}^\pm) \leq 0$ for any second eigenfunction $\varphi_{2,p}$ corresponding to  $\lambda_2(p)$. 
	Thus, there exists a minimizing sequence $\{u_n\}_{n \in \mathbb{N}} \subset \mathcal{B}_{\mathcal{L}}(\alpha)$ for $\beta_\mathcal{L}(\alpha)$.
	Consider the corresponding normalized sequence $\{v_n\}_{n \in \mathbb{N}} \subset \mathcal{B}_{\mathcal{L}}(\alpha)$ given by \eqref{normlizedseq}.
	Lemma~\ref{lem:normalized1} implies that the limit point $v_0 \in \mathcal{B}_{\mathcal{L}}(\alpha)$, and hence
	$$
	\beta_\mathcal{L}(\alpha) \leq \min \left\{ \frac{\intO |\nabla v_0^+|^q \, dx}{\intO |v_0^+|^q \, dx}, \frac{\intO |\nabla v_0^-|^q \, dx}{\intO |v_0^-|^q \, dx} \right\} \leq \liminf_{n \to +\infty}\,
	\min \left\{ \frac{\intO |\nabla v_n^+|^q \, dx}{\intO |v_n^+|^q \, dx}, \frac{\intO |\nabla v_n^-|^q \, dx}{\intO |v_n^-|^q \, dx} \right\} = \beta_\mathcal{L}(\alpha),
	$$
	which means that $v_0$ is a minimizer of $\beta_\mathcal{L}(\alpha)$. 
\end{proof}

\begin{remark}
The definition \eqref{bK} of $\beta_\mathcal{L}(\alpha)$ is equivalent to
\begin{align}\label{betaK1}
\beta_{\mathcal{L}}(\alpha) 
:=\inf \left\{\,
\frac{\intO |\nabla u^+|^q \, dx}{\intO |u^+|^q \, dx} 
:~ u \in \mathcal{B}_{\mathcal{L}}(\alpha) \right\}.
\end{align}
This can be seen by testing $\beta_\mathcal{L}(\alpha)$ either with the corresponding minimizer $u_{\alpha}$ or with $-u_{\alpha}$.
\end{remark} 

Consider now the critical value 
\begin{equation}\label{eq:bL*}
\beta_\mathcal{L}^*:=\inf\left\{
\frac{\intO |\nabla \varphi^+|^q \, dx}{\intO |\varphi^+|^q \, dx}:~ 
\varphi\in ES(p,\lambda_2(p))\setminus\{0\}\right\}.
\end{equation}
The following lemma contains the main properties of $\beta_{\mathcal{L}}(\alpha)$.
\begin{lemma}\label{lem:bK}
	The following assertions hold:
	\begin{enumerate}
		\item[{\rm (i)}] $\beta_{\mathcal{L}}(\alpha) = +\infty$  for any $\alpha < \lambda_2(p)$, and  $\beta_{\mathcal{L}}(\alpha) \in (\lambda_1(q),\beta_{\mathcal{L}}^*]$ for any $\alpha \geq \lambda_2(p)$;
		\item[{\rm (ii)}] $\beta_{\mathcal{L}}(\alpha)$ is nonincreasing for $\alpha \in [\lambda_2(p), +\infty)$;
		\item[{\rm (iii)}] $\beta_{\mathcal{L}}(\alpha)$ is right-continuous for $\alpha \in [\lambda_2(p), +\infty)$; 
		\item[{\rm (iv)}]
		$\mathcal{K}_{\alpha, \beta} \not= \emptyset$ if and only if $\alpha \ge \lambda_2(p)$ and $\beta \ge \beta_\mathcal{L}(\alpha)$, where $\mathcal{K}_{\alpha, \beta}$ is defined by 
		\begin{align}
		\mathcal{K}_{\alpha,\beta} 
		&:= \{
		u \in W_0^{1,p}:~ u^\pm \not\equiv 0, ~H_\alpha(u^+) \leq 0,~
		H_\alpha(u^-) \leq 0, ~ G_\beta(u^+) \leq 0 \, \}
		\label{def:K} \\ 
		&~ =\mathcal{B}_{\mathcal{L}}(\alpha)\cap\{u \in W_0^{1,p}:~ G_\beta(u^+) \leq 0\}. 
		\nonumber
		\end{align} 
	\end{enumerate}
\end{lemma}
\begin{proof}
(i) As stated in the proof of Lemma~\ref{lem:easy-fact2}, we easily see that 
$\mathcal{B}_{\mathcal{L}}(\alpha) = \emptyset$ for all $\alpha < \lambda_2(p)$, and hence $\beta_\mathcal{L}(\alpha) = +\infty$.  
If $\alpha\ge \lambda_2(p)$, then $ES(p,\lambda_2(p))\setminus\{0\}\subset \mathcal{B}_{\mathcal{L}}(\alpha)$, and using \eqref{betaK1} we obtain that  $\beta_{\mathcal{L}}(\alpha)\le \beta_{\mathcal{L}}^*$. 
Since any sign-changing function $w \in \W$ satisfies $\|\nabla w^\pm\|_q^q>\lambda_1(q)\|w^\pm\|_q^q$ (see Lemma~\ref{lem:1}), 
taking a minimizer $u_\alpha$ of $\beta_{\mathcal{L}}(\alpha)$ 
(see Proposition~\ref{prop:exist-minimizer}), 
we conclude that  $\beta_{\mathcal{L}}(\alpha)=
\|\nabla u_\alpha^+\|_q^q/\|u_\alpha^+\|_q^q >\lambda_1(q)$ for all $\alpha \geq \lambda_2(p)$. 

The assertion (ii) can be proved as in Lemma~\ref{lem:b1}. 

(iii) Due to the assertion (ii), it is sufficient to show that 
$\beta_{\mathcal{L}}(\alpha_0) \leq 
\lim\limits_{\alpha \to \alpha_0+0}\beta_{\mathcal{L}}(\alpha)$ for all $\alpha_0 \geq \lambda_2(p)$. 
Since $\beta_\mathcal{L}(\alpha)$ is monotone and bounded in a right neighborhood of $\alpha_0$, for any decreasing sequence $\{\alpha_n\}_{n \in \mathbb{N}}$ such that $\alpha_n \to \alpha_0+0$ as $n \to +\infty$ there holds $\lim\limits_{n \to +\infty}\beta_{\mathcal{L}}(\alpha_n) = \lim\limits_{\alpha \to \alpha_0+0}\beta_{\mathcal{L}}(\alpha)$. 
According to Proposition~\ref{prop:exist-minimizer}, 
for each $n \in \mathbb{N}$ there exists a minimizer $u_n\in \mathcal{B}_{\mathcal{L}}(\alpha_n)$ of 
$\beta_{\mathcal{L}}(\alpha_n)$, and we can assume that $\|\nabla u_n^\pm\|_p=1$. 
Thus, passing to an appropriate subsequence, $u_n$ converges to some $u_0 \in \W$ weakly in $\W$ and strongly in $L^p(\Omega)$. 
Moreover, $u_0^\pm \not\equiv 0$ in $\Omega$, since $H_{\alpha_n}(u_n^\pm) \leq 0$ implies that $\|u_n^\pm\|_p^p \geq 1/\alpha_n$. 
Furthermore, due to the weak lower semicontinuity of the $\W$-norm, 
we have 
$H_{\alpha_0}(u_0^\pm) \leq 0$, and hence 
$u_0\in \mathcal{B}_{\mathcal{L}}(\alpha_0)$. 
Consequently, using \eqref{betaK1}, we conclude that  
\begin{align*}
\beta_{\mathcal{L}}(\alpha_0) 
\leq 
\frac{\intO |\nabla u_0^+|^q \, dx}{\intO |u_0^+|^q \, dx}
\leq 
\liminf_{n \to +\infty}
\frac{\intO |\nabla u_n^+|^q \, dx}{\intO |u_n^+|^q \, dx}
= 
\liminf_{n \to +\infty} \beta_{\mathcal{L}}(\alpha_n)
=\lim_{\alpha \to \alpha_0+0}\beta_{\mathcal{L}}(\alpha).
\end{align*}

(iv) 
Assume that $\alpha \geq \lambda_2(p)$ and $\beta \geq \beta_\mathcal{L}(\alpha)$. 
Let $u \in \mathcal{B}_{\mathcal{L}}(\alpha)$ be a minimizer of $\beta_\mathcal{L}(\alpha)$. Then $H_\alpha(u^\pm) \leq 0$ and, in view of \eqref{betaK1}, we may suppose that $G_{\beta_\mathcal{L}(\alpha)}(u^+)=0$. 
Therefore, $G_{\beta}(u^+) \leq G_{\beta_\mathcal{L}(\alpha)}(u^+)=0$ and hence $u \in \mathcal{K}_{\alpha,\beta}$.

Suppose now that there exists $u \in \mathcal{K}_{\alpha, \beta}$ for some $\alpha, 
\beta \in \mathbb{R}$. 
Since $\mathcal{K}_{\alpha, \beta} \subset \mathcal{B}_\mathcal{L}(\alpha)$, the assertion (i) implies that $\alpha \geq \lambda_2(p)$. 
Moreover, since $G_\beta(u^+) \leq 0$, the definition of $\beta_{\mathcal{L}}(\alpha)$ leads to 
$$
\beta_{\mathcal{L}}(\alpha) \leq 
\min \left\{ \frac{\intO |\nabla u^+|^q \, dx}{\intO |u^+|^q \, dx}, 
\frac{\intO |\nabla u^-|^q \, dx}{\intO |u^-|^q \, dx} \right\}
\le 
\frac{\intO |\nabla u^+|^q \, dx}{\intO |u^+|^q \, dx} \le \beta,
$$
which completes the proof.
\end{proof}

In the sequel, it will be convenient to use the notation
\begin{equation}\label{sigmaK}
\Sigma_{\mathcal{L}} := \{ (\alpha, \beta) \in \mathbb{R}^2: \alpha >  \lambda_2(p), ~ \beta < \beta_{\mathcal{L}}(\alpha) \}.
\end{equation}
\begin{remark}\label{remSigmaK}
Due to Lemmas~\ref{lem:b1} and \ref{lem:bK}, 
the definitions of 
$\beta_1^*$ and $\beta_{\mathcal{L}}^*$ 
(see \eqref{def:beta_1} and \eqref{eq:bL*}) 
imply that 
$\beta_{\mathcal{L}}(\alpha) \leq
\beta_{\mathcal{L}}^* \leq
\beta_1^* \leq
\beta_1(\alpha) 
$ 
for all $\alpha > \lambda_2(p)$, and hence $\mathcal{M}_{\alpha, \beta}^1 \neq \emptyset$ for any $(\alpha,\beta) \in \Sigma_{\mathcal{L}}$. 
\end{remark}

\begin{remark}
In the one-dimensional case we have
\begin{equation}\label{remark:beta:N=1}
\beta_1^*=\beta_{\mathcal{L}}^*\in (\lambda_2(q), \lambda_4(q)). 
\end{equation}
Indeed, if $\Omega=(0,T)$, then the second eigenfunction $\varphi_{2,p}$ is given explicitly through the first eigenfunction $\varphi_p$ by 
$\varphi_{2, p}(x) = \varphi_p(2x)$ 
for $x \in (0, T/2]$, and $\varphi_{2,p}(x) = -\varphi_p(2x-T)$ 
for $x \in (T/2, T)$ (see Appendix~A).
Hence, Lemma~\ref{lem:appendixA} in Appendix~A implies that 
$$
\frac{\|(\varphi_{2,p}^+)'\|_q^q}{\|\varphi_{2,p}^+\|_q^q}
=2^q\frac{\|\varphi_p'\|_q^q}{\|\varphi_p\|_q^q}\in 
(2^q\lambda_1(q), 2^q\lambda_2(q))
=(\lambda_2(q), \lambda_4(q)),
$$
and, consequently,  \eqref{remark:beta:N=1} holds. 
\end{remark}

\subsection{Existence of positive energy nodal solutions}
In this subsection, we prove the existence of nodal solutions in the set $\Sigma_{\mathcal{L}}$ defined by \eqref{sigmaK}. 
To this end, we consider the minimization of the energy functional $E_{\alpha, \beta}$ over the set $\mathcal{M}_{\alpha, \beta}^1$. 

First, we prepare the following auxiliary lemma. 

\begin{lemma}\label{bdd-minimizing} 
Let $\{u_n\}_{n \in \mathbb{N}}$ be an arbitrary sequence in $\mathcal{M}_{\alpha,\beta}^1$ and let $\{v_n\}_{n \in \mathbb{N}} \subset \mathcal{M}_{\alpha,\beta}^1$ be a corresponding normalized sequence given by \eqref{normlizedseq} in Lemma~\ref{lem:normalized1}. 
If $\|\nabla u_n^+\|_p \to +\infty$ as $n \to  +\infty$, and 
$\{E_{\alpha,\beta}(u_n^+)\}_{n \in \mathbb{N}}$ is bounded from above, then $G_\beta(v_0^+) \leq 0$. Consequently, $v_0 \in \mathcal{K}_{\alpha, \beta}$.
\end{lemma} 
\begin{proof} 
Assume that $\{E_{\alpha,\beta}(u_n^+)\}_{n \in \mathbb{N}}$ is 
bounded from above. 
Recalling that $-G_\beta(u_n^\pm)=H_\alpha (u_n^\pm) < 0$ 
by $u_n\in\mathcal{M}^1_{\alpha,\beta}$ and noting that the equalities
\begin{equation}\label{E>0}
E_{\alpha, \beta}(u) = \frac{p-q}{pq} G_\beta(u) = 
-\frac{p-q}{pq}  H_\alpha(u)
\end{equation}
hold for all $u \in \mathcal{N}_{\alpha, \beta}$, we get the boundedness of $G_\beta(u_n^+)$: 
$$
0<\frac{p-q}{pq} G_\beta(u_n^+)
=E_{\alpha, \beta}(u_n^+)\le 
\sup_{l\in\mathbb{N}} E_{\alpha,\beta}(u_l^+)<+\infty.
$$
Consequently, the weak lower semicontinuity and the assumption that $\|\nabla u_n^+\|_p \to +\infty$ as $n \to +\infty$ imply
$$
G_\beta(v_0^+) \leq \liminf_{n \to +\infty} G_\beta(v_n^+) 
=\liminf_{n \to +\infty} \frac{G_\beta(u_n^+)}{\|\nabla u_n^+\|_p^q}=0.
$$
Combining this inequality with the fact that $v_0 \in \mathcal{B}_\mathcal{L}(\alpha)$ (see Lemma~\ref{lem:normalized1}), we conclude that $v_0 \in \mathcal{K}_{\alpha, \beta}$.
\end{proof} 

From Remark~\ref{remSigmaK} we know that $\mathcal{M}_{\alpha, \beta}^1 \neq \emptyset$ for any $(\alpha, \beta) \in \Sigma_\mathcal{L}$. Hence, there exists a minimizing sequence for $E_{\alpha, \beta}$ over $\mathcal{M}_{\alpha, \beta}^1$.
Moreover, this minimizing sequence, in fact, converges.

\begin{theorem}\label{thm:minimizer}
	Let $(\alpha, \beta) \in \Sigma_{\mathcal{L}}$. Then there exists a minimizer $u \in \mathcal{M}_{\alpha, \beta}^1$ of $E_{\alpha, \beta}$ over ${\mathcal{M}_{\alpha, \beta}^1}$.
\end{theorem}
\begin{proof}
Assume $\{u_n\}_{n \in \mathbb{N}} \subset \mathcal{M}_{\alpha, \beta}^1$ to be a minimizing sequence for $E_{\alpha, \beta}$ over $\mathcal{M}^1_{\alpha,\beta}$. 
Equalities \eqref{E>0} imply that $E_{\alpha, \beta}(u_n^\pm) > 0$, and hence 
$\{E_{\alpha, \beta}(u_n)\}_{n \in \mathbb{N}}$ and 
$\{E_{\alpha, \beta}(u_n^\pm)\}_{n \in \mathbb{N}}$ are bounded.  
Applying Lemma~\ref{bdd-minimizing}, we conclude that if $\|\nabla u_n^+\|_p \to +\infty$ as $n \to +\infty$, then the limit $v_0$ of a normalized sequence \eqref{normlizedseq} belongs to the set $\mathcal{K}_{\alpha, \beta}$ defined by \eqref{def:K}.
However, $\mathcal{K}_{\alpha, \beta} = \emptyset$ for all $(\alpha, \beta) \in \Sigma_{\mathcal{L}}$, due to  Lemma~\ref{lem:bK} (iv). 
This is a contradiction. 
Thus, $\{u_n^+\}_{n \in \mathbb{N}}$ is bounded in $\W$. 
Since $\{-u_n\}_{n \in \mathbb{N}}$ is also a minimizing sequence for $E_{\alpha, \beta}$ over $\mathcal{M}_{\alpha, \beta}^1$, we apply the same arguments to derive that $(-u_n)^+ \equiv u_n^-$ is bounded in $\W$, which finally yields the boundedness of the whole sequence $\{u_n\}_{n \in \mathbb{N}}$.

Let us now show that $\|\nabla u_n^+\|_p$ and $\|\nabla u_n^-\|_p$ 
do not converge to zero. 
Applying the assertions (ii) and (iii) of Lemma~\ref{lem:normalized1} to the corresponding normalized sequence $\{v_n\}_{n \in \mathbb{N}}$ given by \eqref{normlizedseq}, 
we see that its limit point $v_0$ belongs to $\mathcal{B}_\mathcal{L}(\alpha)$. 
Suppose, by contradiction, that $\|\nabla u_n^+\|_p \to 0$ as $n \to +\infty$.
Then, using the Nehari constraints, we get
\begin{equation*}
0<G_\beta(v_n^+) =-\|\nabla u_n^+\|_p^{p-q} H_\alpha(v_n^+) 
\to 0 
\quad \text{as} \quad
n \to +\infty,
\end{equation*}
since $H_\alpha$ is bounded on a bounded set and $\|\nabla v_n^+\|_p=1$. 
Consequently, $G_\beta(v^+_0) \leq \liminf\limits_{n \to +\infty}G_\beta(v^\pm_n)=0$ and 
$H_\alpha(v^\pm_0) \leq \liminf\limits_{n \to +\infty}H_\alpha(v^\pm_n) \leq 0$, 
i.e., $v_0 \in \mathcal{K}_{\alpha, \beta}$, 
and we obtain a contradiction as above. 
In the case $\|\nabla u_n^-\|_p \to 0$, we
consider $-u_n$ instead of $u_n$, and again obtain a contradiction. 
As a result, there hold
\begin{equation}\label{eq:not-to-zero}
\delta^+:=\inf_{n\in \mathbb{N}}\|\nabla u_n^+\|_p^p>0 \quad \text{and} \quad 
\delta^-:=\inf_{n\in \mathbb{N}}\|\nabla u_n^-\|_p^p>0.
\end{equation}
Now, choosing an appropriate subsequence, we get $u_n \rightharpoonup u_0$ weakly in $W_0^{1,p}$ and $u_n \to u_0$ strongly in $L^p(\Omega)$, where $u_0 \in \W$. 
Inequalities \eqref{eq:not-to-zero} together with $H_\alpha(u_n^\pm)<0$ imply that 
$\|u_n^\pm\|_p^p\ge \delta^\pm/\alpha$ for all $n \in \mathbb{N}$, and hence $u_0^\pm \not\equiv 0$. 
At the same time, the weak lower semicontinuity yields
\begin{equation}\label{eq:H<0} 
H_\alpha(u_0^\pm)\le \liminf_{n \to +\infty} H_\alpha(u_n^\pm) \leq 0.
\end{equation} 
Let us show that 
\begin{equation}\label{eq:H<0<G} 
H_\alpha(u_0^+)<0<G_\beta(u_0^+)\quad {\rm and}\quad 
H_\alpha(u_0^-)<0<G_\beta(u_0^-). 
\end{equation}
Indeed, since $\mathcal{K}_{\alpha,\beta}$ is empty for 
$(\alpha,\beta)\in \Sigma_{\mathcal{L}}$, we see that 
$u_0^+ - u_0^-\not\in \mathcal{K}_{\alpha,\beta}$ 
and $u_0^- -u_0^+\not\in \mathcal{K}_{\alpha,\beta}$. 
This leads to $G_\beta(u_0^\pm)>0$, since $H_\alpha(u_0^\pm)\le 0$  by \eqref{eq:H<0}. 
Finally, from the Nehari constraints and the weak lower semicontinuity we derive that 
\begin{equation*}
H_\alpha(u_0^\pm) + G_\beta(u_0^\pm) 
\leq \liminf_{n \to +\infty} 
\left( H_\alpha(u_n^\pm) + G_\beta(u_n^\pm) \right) =  0.
\end{equation*}
This means that $H_\alpha(u_0^\pm) \leq -G_\beta(u_0^\pm) <0$, 
and hence \eqref{eq:H<0<G} is shown. 

According to \eqref{eq:H<0<G}, 
Lemma~\ref{lemm:E>0} implies the existence of unique maximum points 
$t_0^+>0$ of $E_{\alpha, \beta}(t u_0^+)$ and $t_0^->0$ of  $E_{\alpha, \beta}(t u_0^-)$ with respect to $t > 0$, and 
$t_0^\pm u_0^\pm \in \mathcal{N}_{\alpha, \beta}$.  Accordingly, we conclude from  \eqref{eq:H<0<G} that $t_0^+ u_0^+ - t_0^- u_0^- \in \mathcal{M}_{\alpha,\beta}^1$. 
Therefore, 
\begin{align*}
\inf_{\mathcal{M}_{\alpha, \beta}^1} E_{\alpha, \beta}
\le E_{\alpha, \beta}(t_0^+ u_0^+ - t_0^- u_0^-) 
&\leq \liminf_{n \to +\infty} E_{\alpha, \beta}(t_0^+ u_n^+ - t_0^- u_n^-) 
\\ 
&= 
\liminf_{n \to +\infty} (E_{\alpha, \beta}(t_0^+ u_n^+) + E_{\alpha, \beta}(t_0^- u_n^-))\\
&\le \liminf_{n \to +\infty} (E_{\alpha, \beta}(u_n^+) + E_{\alpha, \beta}(u_n^-)) 
=\liminf_{n \to +\infty} E_{\alpha, \beta}(u_n)
= \inf_{\mathcal{M}_{\alpha, \beta}^1} E_{\alpha, \beta}. 
\end{align*}
The last inequality in this formula is due to the fact that 
$\max\limits_{t>0}E_{\alpha, \beta}(t u_n^\pm) =E_{\alpha, \beta}(u_n^\pm)$, see  Lemma~\ref{lemm:E>0}. 
Consequently, $t_0^+u_0^+- t_0^-u_0^- \in \mathcal{M}_{\alpha, \beta}^1$ 
is the minimizer of $E_{\alpha, \beta}$ over $\mathcal{M}_{\alpha, \beta}^1$.
\end{proof}

\begin{lemma}\label{lem:solution}
	Let $(\alpha, \beta) \in \Sigma_{\mathcal{L}}$. If $u \in \mathcal{M}_{\alpha, \beta}^1$ is a minimizer of $E_{\alpha, \beta}$ over $\mathcal{M}_{\alpha, \beta}^1$, then $u$ is a critical point of $E_{\alpha, \beta}$ on $\W$.
\end{lemma}
\begin{proof}
The proof can be handled in much the same way as the proof of \cite[Lemma~3.2]{BobkovEJQT2014}, where a variant of the deformation lemma was used in a framework of the problem with indefinite nonlinearities; see also \cite[Proposition~3.1]{bartschwillem}.
\end{proof}

\subsection{Qualitative properties}
In this subsection we show that 
\textit{any} minimizer $u$ of $E_{\alpha, \beta}$ over ${\mathcal{M}_{\alpha, \beta}^1}$ for $(\alpha,\beta)\in\Sigma_\mathcal{L}$ 
has exactly two nodal domains (that is, connected components of $\Omega \setminus u^{-1}(0)$). 

\begin{lemma}\label{lem:twonodal}
	Let $(\alpha, \beta) \in \Sigma_\mathcal{L}$ and let $u \in \mathcal{M}_{\alpha, \beta}^1$ be a minimizer 
	of $E_{\alpha, \beta}$ over ${\mathcal{M}_{\alpha, \beta}^1}$. Then $u$ has exactly two nodal domains.
\end{lemma}
\begin{proof}
	Suppose, contrary to our claim, that there exists a minimizer 
	$u\in\mathcal{M}_{\alpha,\beta}^1$ of $E_{\alpha, \beta}$ over ${\mathcal{M}_{\alpha, \beta}^1}$ with (at least) three nodal domains. 	
	We decompose $u$ such that $u = u_1 + u_2 + u_3$, where $u_i \not\equiv 0$ for $i=1,2,3$, and each $u_i$ is of a constant sign on its support. 
	Note that each $u_i \in \mathcal{N}_{\alpha, \beta}$. 
	Indeed, $u_i\in W_0^{1,p}$ (cf. \cite[Lemma 5.6]{cuesta}), and since $u$ is a solution of $(GEV;\alpha,\beta)$ by Lemma~\ref{lem:solution}, we obtain
	\begin{equation}\label{eq:H+G=0}
	0=\left<E_{\alpha, \beta}'(u), u_i \right> = H_\alpha(u_i) + G_\beta(u_i)  
	\quad\text{ for }~  i=1,2,3. 
	\end{equation}
	Assume, without loss of generality, that $u^+ = u_1 + u_2$ and $u^- = -u_3$. Since $u \in \mathcal{M}_{\alpha, \beta}^1$, we have
	\begin{align*}
	H_{\alpha}(u^+)  
	= 
	H_{\alpha}(u_1) + 
	H_{\alpha}(u_2)
	< 0
	~\text{ and }~
	H_{\alpha}(u^-) = 
	H_{\alpha}(-u_3)= H_{\alpha}(u_3)< 0.
	\end{align*} 
	Moreover, we may assume that $H_\alpha(u_2) \leq H_\alpha(u_1)$, whence $H_\alpha(u_2)<0$. This assumption splits into the following four cases:
	
	(i) $H_{\alpha}(u_2) \leq H_{\alpha}(u_1) < 0$; 
	
	(ii) $H_\alpha(u_2)<H_\alpha(u_1)=0$; 
	
	(iii) $H_\alpha(u_2)<0<H_\alpha(u_1)$ and $H_\alpha(u_1)+H_\alpha(u_3) \geq 0$; 
	
	(iv) $H_\alpha(u_2)<0<H_\alpha(u_1)$ and $H_\alpha(u_1)+H_\alpha(u_3) < 0$. 

	\noindent
	Now we will subsequently show a contradiction for each case.

	(i) It is easy to see  that $u_1 + u_3 \in \mathcal{M}_{\alpha, \beta}^1$. 
	Since $H_{\alpha}(u_2) < 0$ leads to $E_{\alpha, \beta}(u_2) > 0$, we have a contradiction by the following inequality: 
	$$
	\inf_{\mathcal{M}_{\alpha, \beta}^1} E_{\alpha, \beta} \leq
	E_{\alpha, \beta}(u_1 + u_3) < 
	E_{\alpha, \beta}(u_1 + u_3) +E_{\alpha, \beta}(u_2)= 
	E_{\alpha, \beta}(u_1 + u_2 + u_3) = \inf_{\mathcal{M}_{\alpha, \beta}^1} E_{\alpha, \beta}. 
	$$
	
	(ii) Since $H_\alpha(u_1) = 0$, we can derive from \eqref{eq:H+G=0} that  $G_\beta(u_1) = 0$. Recalling that $H_\alpha(u_2) < 0$, we get $u_1-u_2 \in \mathcal{K}_{\alpha,\beta}$, which contradicts the assertion (iv) of Lemma~\ref{lem:bK}.
	
	(iii) Recall $H_\alpha(u_3)<0$ and set 
	$$
	1 \le t_0^p:=-\frac{H_\alpha(u_1)}{H_\alpha(u_3)}=-\dfrac{G_\beta(u_1)}{G_\beta(u_3)}. 
	$$
	Since $u_1,u_3\in\mathcal{N}_{\alpha,\beta}$, we obtain 
	$$
	H_\alpha(u_1 - t_0 u_3)  = 
	H_\alpha(u_1) + t_0^p H_\alpha(u_3) =G_\beta(u_1) + t_0^p G_\beta(u_3) = 0.
	$$
	On the other hand, since $G_\beta(u_3) > 0$, $t_0 \ge 1$, and $p>q$, we have
	$$
	0 = G_\beta(u_1) + t_0^p G_\beta(u_3) \ge 
	G_\beta(u_1) + t_0^q G_\beta(u_3) = 
	G_\beta(u_1 - t_0 u_3).
	$$
	Consequently, $H_\alpha(u_1 - t_0 u_3) = 0$ and $G_\beta(u_1 - t_0 u_3) \le 0$. Considering a function $w = u_1 - t_0 u_3 - u_2$, we get $w^+ =u_1 - t_0 u_3$ and $w^- = u_2$, which implies that $w \in \mathcal{K}_{\alpha,\beta}$. This is again a contradiction to the emptiness of $\mathcal{K}_{\alpha,\beta}$.
	
	(iv) Consider a function $w = u_1 - u_3 - u_2$. Then
	$w^+ = u_1 - u_3$ and $w^- = u_2$. By the assumptions, we have $H_\alpha(w^\pm) < 0$.
	Therefore, $w \in \mathcal{M}_{\alpha, \beta}^1$ and 
	$$
	E_{\alpha, \beta}(w) = 
	E_{\alpha, \beta}(u_1 - u_3 - u_2) = E_{\alpha, \beta}(u_1) + 
	E_{\alpha, \beta}(u_3) + 
	E_{\alpha, \beta}(u_2) = 
	E_{\alpha, \beta}(u) = \inf_{\mathcal{M}_{\alpha, \beta}^1}E_{\alpha, \beta},
	$$
	that is, $w$ is also a minimizer of $E_{\alpha, \beta}$ over ${\mathcal{M}_{\alpha, \beta}^1}$ and hence a weak solution of $(GEV;\alpha,\beta)$ in view of  Lemma~\ref{lem:solution}. 
	This implies that for any $\xi \in \W$ there holds
	\begin{equation}
	\begin{aligned}
	\label{eq:weak1}
	\int_{\Omega} |\nabla w|^{p-2} \nabla (u_1 - u_3 - u_2) \nabla \xi \, dx
	+\int_{\Omega} |\nabla w|^{q-2} \nabla (u_1 - u_3 - u_2) \nabla \xi \, dx 
	\\
	=\alpha \int_{\Omega} |w|^{p-2} (u_1 - u_3 - u_2) \xi \, dx 
	+\beta \int_{\Omega} |w|^{q-2} (u_1 - u_3 - u_2) \xi \, dx.
	\end{aligned}
	\end{equation}
	On the other hand, since $u=u_1+u_2+u_3$ is also a weak solution of $(GEV;\alpha,\beta)$, we obtain 
	\begin{equation}
	\begin{aligned}
	\label{eq:weak2}
	\int_{\Omega} |\nabla u|^{p-2} \nabla (u_1 + u_3 + u_2) \nabla \xi \, dx
	+\int_{\Omega} |\nabla u|^{q-2} \nabla (u_1 + u_3 + u_2) \nabla \xi \, dx
	\\
	=\alpha \int_{\Omega} |u|^{p-2} (u_1 + u_3 + u_2) \xi \, dx 
	+\beta \int_{\Omega} |u|^{q-2} (u_1 + u_3 + u_2) \xi \, dx
	\end{aligned}
	\end{equation}
	for all $\xi \in \W$. 
	Summarizing \eqref{eq:weak1} and \eqref{eq:weak2} and noting that $|u| \equiv |w|$ and $|\nabla u| \equiv |\nabla w|$, we get
	\begin{align*}
	\int_{\Omega} |\nabla u_1|^{p-2} \nabla u_1 \nabla \xi \, dx
	+\int_{\Omega} |\nabla u_1|^{q-2} \nabla u_1  \nabla \xi \, dx
	=\alpha \int_{\Omega} |u_1|^{p-2} u_1 \xi \, dx 
	+\beta \int_{\Omega} |u_1|^{q-2} u_1 \xi \, dx
	\end{align*}
	for each $\xi \in \W$, since the supports of $u_i$ are mutually disjoint.
	This means that $u_1$ is a nonnegative solution of $(GEV;\alpha,\beta)$ \textit{in} $\Omega$. 
	However, the strong maximum principle implies that $u_1 > 0$ in $\Omega$, cf. \cite[Remark~1, p.~3284]{BobkovTanaka2015}. Hence, $u_2\equiv 0$ and $u_3\equiv 0$, which is a contradiction. 
\end{proof}

\section{Nodal solutions with negative energy}\label{sec:minimax} 
In this section, we provide the main ingredients for the proofs of Theorems~\ref{thm:neg1} and \ref{thm:neg2}.

\subsection{Auxiliary results}
Consider the set
\begin{equation*}
Y(\lambda):=\{u\in\W:~ \|\nabla u\|_p^p \geq \lambda\|u\|_p^p \,\},
\end{equation*}
where $\lambda \geq 0$. 
Hereinafter, by $S^k_+$ we denote the closed unit upper hemisphere in $\mathbb{R}^{k+1}$ with the boundary $S^{k-1}$. 
We begin with the following linking lemma. 
\begin{lemma}\label{lem:link} 
Let	$k \in \mathbb{N}$.
Then $h(S^k_+)\cap Y(\lambda_{k+1}(p)) \neq \emptyset$ for any $h\in C(S^k_+,\W)$  provided $h\big|_{S^{k-1}}$ is odd.
\end{lemma}
\begin{proof} 
Fix any $h\in C(S^k_+,\W)$ such that 
$h\big|_{S^{k-1}}$ is odd. 
If $\|u\|_p=0$ for some $u\in h(S^{k}_+)$, 
then, obviously, $u\in Y(\lambda_{k+1}(p))$. 
Thus, we may assume that $\|u\|_p > 0$ for 
every $u\in h(S^{k}_+)$. 
Define the map
$$
\tilde{h}(z):=
\left\{
\begin{aligned}
&h(z)/\|h(z)\|_p
&& \text{ if } z\in S^k_+, \\
-&h(-z)/\|h(-z)\|_p
&& \text{ if } z\in S^k_-. 
\end{aligned}
\right. 
$$
It is not hard to see that $\tilde{h} \in \mathscr{F}_{k+1}(p)$, where $\mathscr{F}_{k+1}(p)$ is the set given by \eqref{F_n} with $r=p$. 
By the definition \eqref{lambda_n} of $\lambda_{k+1}(p)$, 
there exists 
$z_0\in S^k$ such that 
$\|\nabla \tilde{h}(z_0)\|_p^p \geq \lambda_{k+1}(p)$. Since $\tilde{h}(z_0) \in S(p)$, we have $\|\nabla \tilde{h}(z_0)\|_p^p \geq \lambda_{k+1}(p)\|\tilde{h}(z_0)\|_p^p$. 
Moreover, since $\tilde{h}$ is odd, 
we may suppose that $z_0\in S^k_+$. Consequently, we obtain $h(z_0)\in Y(\lambda_{k+1}(p))$. 
\end{proof} 

\begin{lemma}\label{lem:bdd-below} 
	Let $\alpha, \beta \in \mathbb{R}$ and let $\lambda>\max\{\alpha,0\}$. 
	Then $\E$ is bounded from below on $Y(\lambda)$. 
\end{lemma} 
\begin{proof} 
Assume that $u\in Y(\lambda)$ with $\lambda>\max\{\alpha,0\}$. 
Using the H\"older inequality, we obtain
\begin{align*} 
\E(u) &\ge \frac{\lambda-\alpha}{p\lambda}\|\nabla u\|_p^p 
-\frac{\beta}{q} |\Omega|^{\frac{p-q}{p}}  \|u\|_p^q
\geq \frac{\lambda-\alpha}{p\lambda}\|\nabla u\|_p^p 
-\frac{\beta}{q (\lambda_1(p))^{q/p}} |\Omega|^{\frac{p-q}{p}}  
\|\nabla u\|_p^q,
\end{align*} 
which implies the desired conclusion, since $q<p$. 
\end{proof}

\begin{lemma}\label{lem:bdd-PS}
Assume $\alpha, \beta \in \mathbb{R}$ and let $\{u_n\}_{n \in \mathbb{N}}$ be a sequence in $\W$ which satisfies 
$\|\nabla u_n\|_p \to +\infty$ and 
$\E'(u_n)/\|\nabla u_n\|_p^{p-1} \to 0$ in $(\W)^*$ as $n \to +\infty$. 
Then $v_n:=u_n/\|\nabla u_n\|_p$ has a subsequence strongly convergent in $\W$ to some $v_0\in ES(p;\alpha)\setminus\{0\}$, that is, $\alpha\in\sigma(-\Delta_p)$. 
\end{lemma}
\begin{proof} 
Since $\|\nabla v_n\|_p = 1$ for any $n \in \mathbb{N}$, passing to an appropriate subsequence, 
we may assume that $v_n$ converges to some $v_0$ weakly in $\W$ and 
strongly in $L^p(\Omega)$. 
In particular, $\left<H'_\alpha(v_0), v_n \right> \to \left<H'_\alpha(v_0), v_0 \right>$ as $n \to +\infty$. 
Moreover, 
$$
\frac{\left|\left<E_{\alpha, \beta}'(u_n), v_n - v_0 \right>\right|}{\|\nabla u_n\|_p^{p-1}} 
\leq
\frac{\|\E'(u_n)\|_{(\W)^*}}{\|\nabla u_n\|_p^{p-1}} 
\|\nabla(v_n-v_0)\|_p 
\leq
2 \frac{\|\E'(u_n)\|_{(\W)^*}}{\|\nabla u_n\|_p^{p-1}} 
\to 0
$$
as $n \to +\infty$, by the assumption.
Using these facts, we get 
\begin{align*} 
o(1)&=\left< \frac{\E'(u_n)}{\|\nabla u_n\|_p^{p-1}}-H'_\alpha(v_0), 
v_n-v_0\right>
\\
&=\intO \left(|\nabla v_n|^{p-2}\nabla v_n-|\nabla v_0|^{p-2}\nabla v_0
\right)
(\nabla v_n-\nabla v_0)\,dx 
-\alpha\intO \left(|v_n|^{p-2}v_n-|v_0|^{p-2}v_0\right)
(v_n-v_0)\,dx 
\\
&\quad +\frac{1}{\|\nabla u_n\|_p^{p-q}}\intO |\nabla v_n|^{q-2}\nabla v_n
(\nabla v_n-\nabla v_0)\,dx 
-\frac{\beta}{\|\nabla u_n\|_p^{p-q}}\intO |v_n|^{q-2}v_n(v_n-v_0)\,dx
\\
&=\intO \left(|\nabla v_n|^{p-2}\nabla v_n-|\nabla v_0|^{p-2}\nabla v_0\right)
(\nabla v_n-\nabla v_0)\,dx +o(1)
\\
&\ge \left(\|\nabla v_n\|_p^{p-1}-\|\nabla v_0\|_p^{p-1}\right) 
\left(\|\nabla v_n\|_p-\|\nabla v_0\|_p\right)+o(1), 
\end{align*}
where the last inequality is obtained by H\"older's inequality. 
Hence, $\|\nabla v_n\|_p \to \|\nabla v_0\|_p = 1$ as $n \to +\infty$, 
and the uniform convexity of $\W$ implies that $v_n$ converges to $v_0$ strongly in $\W$. 

On the other hand, for any $\xi \in\W$ the following equality holds:
\begin{align*} 
\left< \frac{\E'(u_n)}{\|\nabla u_n\|_p^{p-1}}, \xi \right>
&=\intO |\nabla v_n|^{p-2}\nabla v_n\nabla \xi\,dx 
-\alpha\intO |v_n|^{p-2}v_n \xi\,dx 
\\
&\quad +\frac{1}{\|\nabla u_n\|_p^{p-q}}\intO |\nabla v_n|^{q-2}\nabla v_n
\nabla \xi\,dx 
-\frac{\beta}{\|\nabla u_n\|_p^{p-q}}\intO |v_n|^{q-2}v_n \xi\,dx.
\end{align*}
Therefore, passing to the limit as $n \to +\infty$, we derive
$$
\intO |\nabla v_0|^{p-2}\nabla v_0\nabla \xi\,dx 
-\alpha\intO |v_0|^{p-2}v_0 \xi\,dx=0 
$$
for all $\xi \in \W$, that is, 
$v_0 \in  ES(p;\alpha)\setminus\{0\}$. 
\end{proof}

\begin{lemma}\label{PS} 
If $\alpha\not\in\sigma(-\Delta_p)$, then $\E$ satisfies 
the Palais--Smale condition. 
\end{lemma}
\begin{proof} 
Let $\{u_n\}_{n \in \mathbb{N}} \subset \W$ be a Palais--Smale sequence for $\E$, that is, 
$$
\E(u_n) \to c 
\quad
\text{and}
\quad 
\|\E'(u_n)\|_{(\W)^*} \to 0 
$$
as $n \to +\infty$, where $c$ is a constant. 
Due to the $(S_+)$-property for the operator $-\Delta_p - \Delta_q$ 
(see Remark~\ref{remark:S_+} below), 
it is sufficient to show that 
$\{u_n\}_{n \in \mathbb{N}}$ is bounded in $\W$. 
If we suppose, by contradiction, that $\|\nabla u_n\|_p \to +\infty$ as $n \to +\infty$, then Lemma~\ref{lem:bdd-PS} implies that 
$\alpha\in\sigma(-\Delta_p)$, which contradicts the assumption of the lemma.
\end{proof} 

\begin{remark}\label{remark:S_+} 
For the reader's convenience we show that 
the operator $-\Delta_p - \Delta_q$ has the $(S_+)$-property, namely, 
any sequence $\{u_n\}_{n\in\mathbb{N}}\subset W_0^{1,p}$ converging to some $u_0$ 
weakly in $\W$ and satisfying  
\begin{equation}\label{eq:S_+}
\limsup_{n \to +\infty}\left< -\Delta_p u_n-\Delta_q u_n, u_n-u_0\right> \leq 0, 
\end{equation}
converges strongly in $\W$. Let $u_n \rightharpoonup u_0$ in $\W$ as $n \to +\infty$, and let \eqref{eq:S_+} holds. 
Then the H\"older inequality yields 
\begin{align*}
\left<-\Delta_p u_n-\Delta_q u_n, u_n-u_0\right>+o(1)
=&\left<-\Delta_p u_n-\Delta_q u_n, u_n-u_0\right>
-\left<-\Delta_p u_0-\Delta_q u_0, u_n-u_0\right> = 
\\
\intO \left(|\nabla u_n|^{p-2}\nabla u_n-|\nabla u_0|^{p-2}\nabla u_0\right)
&(\nabla u_n-\nabla u_0)\,dx +
\\
&\intO \left(|\nabla u_n|^{q-2}\nabla u_n-|\nabla u_0|^{q-2}\nabla u_0\right)
(\nabla u_n-\nabla u_0)\,dx
\geq 
\\
\left(\|\nabla u_n\|_p^{p-1}-\|\nabla u_0\|_p^{p-1}\right)
(\|\nabla u_n\|_p&-\|\nabla u_0\|_p)
+\left(\|\nabla u_n\|_q^{q-1}-\|\nabla u_0\|_q^{q-1}\right)
(\|\nabla u_n\|_q-\|\nabla u_0\|_q)
\ge 0, 
\end{align*}
which implies that $\|\nabla u_n\|_p\to \|\nabla u_0\|_p$  
and $\|\nabla u_n\|_q\to \|\nabla u_0\|_q$ as $n \to +\infty$. 
Due to the uniform convexity of $\W$, we conclude that $u_n$ converges to $u_0$ strongly in $\W$. 
\end{remark} 

Recall the definition \eqref{b_*2}:
\begin{equation}\label{b_*21}
\beta^{*}_\mathcal{U}(\alpha) 
:= 
\sup \left\{\,
\frac{\intO |\nabla \varphi|^q \, dx}{\intO |\varphi|^q \, dx}:~ 
\varphi \in ES(p; \alpha)\setminus \{0\} \right\}.
\end{equation}

\begin{lemma}\label{lem:bU:bound1}
If $\alpha\in\sigma(-\Delta_p)$, then 
$\lambda_1(q) \leq \beta^{*}_\mathcal{U}(\alpha) < +\infty$. 
\end{lemma}
\begin{proof}
Let $\alpha\in\sigma(-\Delta_p)$. Recall that 
\cite[Lemma~9]{T-2014} implies the existence of a constant $C(\alpha)>0$ such that $\|\nabla u\|_p \leq C(\alpha) \|u\|_q$ for any $u\in X(\alpha)$, 
where $X(\alpha)$ is defined by \eqref{def:X}. 
Thus, applying the H\"older inequality, we get
	$$
	\intO |\nabla u|^q \, dx \leq |\Omega|^\frac{p-q}{p} \left(\intO |\nabla u|^p \, dx\right)^{q/p} \leq 
	|\Omega|^\frac{p-q}{p} C(\alpha)^q 	\intO |u|^q \, dx
	$$
	for any $u\in X(\alpha)$. Therefore, $\beta^*_\mathcal{U}(\alpha) < +\infty$, since $ES(p; \alpha)\subset X(\alpha)$. 
	On the other hand, it is clear that  $\beta^*_\mathcal{U}(\alpha)\ge \lambda_1(q)$ provided $ES(p; \alpha)\setminus\{0\} \neq \emptyset$. 
\end{proof}

In the one-dimensional case we can clarify the bounds for $\beta^*_\mathcal{U}(\alpha)$ as follows.
\begin{lemma}\label{lem:bU:bound2}
	Let $N=1$ and  $\alpha=\lambda_{k}(p)$, $k \in \mathbb{N}$. Then
	\begin{equation}\label{betaU}
		\lambda_{k+1}(q)\left(\frac{k}{k+1}\right)^q
		= k^q \lambda_1(q)
		<\beta^*_\mathcal{U}(\alpha)=
		k^q
		\frac{\|\varphi_p'\|_q^q}{\|\varphi_p\|_q^q} < k^q \lambda_2(q)
		=\lambda_{k+1}(q)\left(\frac{2 k}{k+1}\right)^q.
	\end{equation}
\end{lemma} 
\begin{proof}
	Let $\Omega = (0, T)$, $T >0$, and  $\alpha = \lambda_k(p)$ for some $k \in \mathbb{N}$. 
	It is known that  $\lambda_k(r) =  (r-1)\left(\frac{k \pi_r}{T}\right)^p$ for any $r>1$ and $k \in \mathbb{N}$ (cf. Appendix~A), and hence the first and third equalities in \eqref{betaU} are satisfied.
	
	Note that the eigenspace $ES(p;\lambda_k(p))$ is one-dimensional, as it follows from \cite[Proposition~2.1]{drabman}. Denoting the corresponding eigenfunction as $\varphi_k$, we directly get $\beta^*_\mathcal{U}(\lambda_k(p))=
	\frac{\|\varphi_k'\|_q^q}{\|\varphi_k\|_q^q}$. On the other hand, $\varphi_k$ has exactly $k$ nodal domains of equivalent length (see Appendix~A), and hence the standard scaling yields $\beta^*_\mathcal{U}(\lambda_k(p))=
	k^q
	\frac{\|\varphi_p'\|_q^q}{\|\varphi_p\|_q^q}$, where $\varphi_p$ is the first eigenfunction of $-\Delta_p$. The inequalities in \eqref{betaU} follow from Lemma~\ref{lem:appendixA} below.
\end{proof}

The following lemma ensues readily from the definition \eqref{b_*21}. 
\begin{lemma}\label{lem:G-negative}
Let $\alpha \in \sigma(-\Delta_p)$ and $\beta>\beta^*_\mathcal{U}(\alpha)$. 
	Then 
$G_\beta(\varphi) < 0$ for all $\varphi \in ES(p;\alpha)\setminus\{0\}$. 
\end{lemma}

\begin{lemma}\label{lem:const-map} 
Let $\alpha \in \mathbb{R}$ and $k \in \mathbb{N}$. 
If $\beta>\lambda_{k+1}(q)$, then there exist an odd map
$h_0\in C(S^k,\W)$ and $t_0>0$ such that 
$$
\max_{z\in S^{k}} \E(t_0 h_0(z)) <0. 
$$
\end{lemma}
\begin{proof} 
Let $\beta>\lambda_{k+1}(q)$ and choose $\varepsilon \in \mathbb{R}$ satisfying 
\begin{equation}\label{eq:map-0}
0<\varepsilon <\frac{1}{2} \quad \text{and} \quad 
\frac{\lambda_{k+1}(q)+2\varepsilon}{(1-2\varepsilon)^q}
<\beta-\varepsilon. 
\end{equation}
By the definition of $\lambda_{k+1}(q)$, 
there exists a map $h_1\in\mathscr{F}_{k+1}(q)$ such that 
\begin{equation}\label{esteps1}
\max_{z\in S^k} \|\nabla h_1(z)\|_q^q <\lambda_{k+1}(q)+\varepsilon.
\end{equation}
Note that by taking $t > 0$ small enough it is easy to get $\max\limits_{z\in S^{k}} \E(t h_1(z)) <0$. However, $h_1 \in C(S^k, S(q))$, and we do not know a priori that $h_1 \in C(S^k, \W)$. Hence the arguments below are needed.

Since $C_0^\infty(\Omega)$ is a dense subset of $W_0^{1,q}$ and $h_1$ is odd, 
for any $z\in S^k$ we can find $u_z \in C_0^\infty(\Omega)$ 
such that 
\begin{equation}\label{eq:map-1} 
u_{-z}=-u_z,
\quad 
\left|\|\nabla h_1(z)\|_q^q - \|\nabla u_z\|_q^q \right| < \varepsilon, \quad \text{and} \quad 
\|h_1(z)-u_z\|_q <\varepsilon. 
\end{equation}
By the continuity of $h_1$, 
for any $z\in S^k$ there exists $\delta(z) \in (0, 1)$ such that 
\begin{equation}\label{eq:map-000} 
\|h_1(z)-h_1(y)\|_q <\varepsilon ~\text{ for all }~ y \in S^k ~\text{ with }~ |z-y|<\delta(z). 
\end{equation}
Considering  $\min\{\delta(z),\delta(-z)\}$ instead of $\delta(z)$, 
we may assume that $\delta$ is even. 
Note that \eqref{eq:map-1} and \eqref{eq:map-000} lead to 
\begin{equation}\label{eq:map-2} 
\|u_z-h_1(y)\|_q < 2\varepsilon ~\text{ for all }~ y \in S^k ~\text{ such that }~ |z-y|<\delta(z).
\end{equation}
Due to the compactness of $S^k$, we may choose a finite number of points $z_i \in S^k$, $i=1,2,\dots,m$, such that 
$$
S^k \subset \bigcup_{i=1}^m \left[B(z_i,\delta(z_i)) 
\cup B(-z_i,\delta(-z_i)) \right],
$$
where $B(z_i,\delta(z_i)) \subset \mathbb{R}^{k+1}$ is a ball of radius $\delta(z_i)$
centered at the point $z_i$. 
Now, for each $i=1,2,\dots,m$ we take a function $\rho_i \in C_0(\mathbb{R}^{k+1})$ such that 
$$
\text{supp } \rho_i = \overline{B(z_i,\delta(z_i))} ~\text{ and }~ 
\rho_i > 0 ~\text{ in }~ B(z_i,\delta(z_i)). 
$$
Note that $B(z_i,\delta(z_i))\cap B(-z_i,\delta(-z_i))=\emptyset$ 
for all $i=1,2,\dots,m$, since  $\delta(z_i)=\delta(-z_i)<1$.
Thus, $\rho_i(-z)=0$ whenever  $\rho_i(z) > 0$.
Define 
$$
\tilde{\rho}_i(z):=\frac{\rho_i(z)}{\sum_{j=1}^m (\rho_j(z)+\rho_j(-z))}
~\text{ for }~
 z \in S^k. 
$$
Since $\left\{B(z_i,\delta(z_i)) 
\cup B(-z_i,\delta(-z_i)) \right\}_{i=1}^m$ is an open covering 
of $S^k$, it is easy to see that $\tilde{\rho}_i \in C(S^k)$  for all $i=1,2,\dots,m$. Moreover, 
\begin{equation}\label{eq:condition}
0\leq \tilde{\rho}_i \leq 1, \quad \tilde{\rho}_i(-z)=0 
~\text{ provided }~ \tilde{\rho}_i(z)>0, 
~\text{ and }~
\sum_{j=1}^m (\tilde{\rho}_j(z)+\tilde{\rho}_j(-z)) = 1 
\end{equation}
for all $z\in S^k$ and $i=1,2,\dots,m$. That is, $\{\tilde{\rho}_i\}_{i=1}^m$ forms a partition of unity of $S^k$.
Set 
$$
h_0(z) := 
\sum_{i=1}^m 
\left(
\tilde{\rho}_i(z)u_{z_i} + \tilde{\rho}_i(-z)u_{-z_i}
\right)
\equiv 
\sum_{i=1}^m 
u_{z_i} \left(
\tilde{\rho}_i(z) - \tilde{\rho}_i(-z)
\right)
~\text{ for }~ z \in S^k. 
$$ 
Evidently, $h_0$ is odd, and the continuity of $\tilde{\rho}_i$ implies that $h_0 \in C(S^k, \W)$. 

Let us show that 
$\max\limits_{z\in S^{k}} \E(t h_0(z)) <0$ for sufficiently small $t > 0$. 
First, for all $z \in S^k$ there holds
\begin{align}\notag
\|\nabla h_0(z)\|_q
&\leq \sum_{i=1}^m 
\|\nabla u_{z_i}\|_q \left(
\tilde{\rho}_i(z)
+\tilde{\rho}_i(-z)
\right) \\
\label{eq:map-3}
&< 
(\lambda_{k+1}(q)+2\varepsilon)^{1/q}
\sum_{i=1}^m (\tilde{\rho}_i(z)+\tilde{\rho}_i(-z)) 
=
(\lambda_{k+1}(q)+2\varepsilon)^{1/q},
\end{align}
where we used that $\|\nabla u_{z_i}\|_q^q < \lambda_{k+1}(q) + 2\varepsilon$, by virtue of \eqref{eq:map-1} and	
\eqref{esteps1}. 
Moreover, 
$h_0(z) \neq 0$ for all $z\in S^k$. Indeed, using the convexity of 
$\|\cdot\|_q^q$, the oddness of $h_1$, 
\eqref{eq:condition} and \eqref{eq:map-2}, we derive 
\begin{align*}
\|h_1(z)-h_0(z)\|_q^q
&=
\|\sum_{i=1}^m 
\left(
\tilde{\rho}_i(z)(h_1(z)-u_{z_i})
+
\tilde{\rho}_i(-z)(h_1(z)-u_{-z_i})
\right)
\|_q^q 
\\ 
&\leq 
\sum_{i=1}^m
\left( \tilde{\rho}_i(z)\|h_1(z)-u_{z_i}\|_q^q 
+
\tilde{\rho}_i(-z)\|u_{z_i} - h_1(-z)\|_q^q
\right)
< 2^q\varepsilon^q, 
\end{align*}
since $\tilde{\rho}_i(-z)>0$ 
if and only if $-z\in B(z_i,\delta(z_i))$. Hence,  $\|h_0(z)\|_q \geq \|h_1(z)\|_q-2\varepsilon=1-2\varepsilon>0$ 
for every $z\in S^k$. 
Now using \eqref{eq:map-3} and \eqref{eq:map-0}, we get
\begin{equation*}
\frac{\|\nabla h_0(z)\|_q^q}{\|h_0(z)\|_q^q} 
<\frac{\lambda_{k+1}(q)+2\varepsilon}{\|h_0(z)\|_q^q}
\le \frac{\lambda_{k+1}(q)+2\varepsilon}{(1-2\varepsilon)^q}
<\beta-\varepsilon
\end{equation*}
for all $z\in S^k$. 
Thus, for sufficiently small $t > 0$ and any $z\in S^k$ we obtain 
\begin{align*} 
\E(t h_0(z)) &=
\frac{t^p}{p}\left(\|\nabla h_0(z)\|_p^p-\alpha\|h_0(z)\|_p^p\right)
+\frac{t^q}{q}\left(\|\nabla h_0(z)\|_q^q-\beta\|h_0(z)\|_q^q\right)
\\ 
&\leq 
\frac{t^p}{p}
\max_{z\in S^k}\left(\|\nabla h_0(z)\|_p^p-\alpha\|h_0(z)\|_p^p\right)
-\frac{t^q (1-2\varepsilon)^q\varepsilon}{q}	<0,
\end{align*}
since $q<p$. This is the desired conclusion.
\end{proof} 

In the sequel, we will also need the following variant of the deformation lemma. 
We refer the reader to \cite[Theorem 3.2]{Chang} for the proof. 
\begin{lemma}\label{retract}
Let $\Psi$ be a $C^1$-functional on a Banach space $W$, let $\Psi$ satisfies 
the Palais--Smale condition at any level $c\in[a,b]$ 
and let $\Psi$ has no critical values in $(a,b)$. 
Assume that either 
$K_a:=\{u \in W:~ \Psi'(u)=0,\ \Psi(u)=a\,\}$ 
consists only of isolated points, or $K_a =\emptyset$. 
Define  
$\Psi^c := \{u \in W:~ \Psi(u) \leq c\,\}$. 
Then, there exists 
$\eta\in C([0,1]\times W, W)$ such that the following hold:
\begin{itemize}
\item[{\rm (i)}] $\Psi(\eta(s,u))$ is nonincreasing in $s$ 
for every $u\in W$;
\item[{\rm (ii)}] $\eta(s,u)=u$\ for any $u\in \Psi^a$, $s \in [0,1]$;
\item[{\rm (iii)}] $\eta(0,u)=u$ and $\eta(1,u)\in \Psi^a$ for any 
$\Psi^b\setminus K_b$;
\item[{\rm (iv)}] if $\Psi$ is even, then $\eta(s,\cdot)$ is odd 
for all $s \in [0,1]$.
\end{itemize}
That is,
$\Psi^a$ is a strong deformation retract 
of $\Psi^b\setminus K_b$. 
\end{lemma}

\subsection{General existence result via minimax arguments} 
In this subsection we prove a result on the existence of an \textit{abstract nontrivial} solution to $(GEV;\alpha,\beta)$. 
Let us emphasize that this result does not guarantee that the obtained solution is  \textit{sign-changing}. 
(However, it is shown in \cite{BobkovTanaka2015} that for sufficiently large 
$\alpha$ and $\beta$ problem $(GEV;\alpha,\beta)$ has 
no sign-constant solutions). 

Recall that we denote $k_\alpha := \min\{k \in \mathbb{N}:~ \alpha < 
\lambda_{k+1}(p)\}$. 
\begin{theorem}\label{thm:general} 
	Assume that $\alpha \in \overline{\mathbb{R}\setminus\sigma(-\Delta_p)}$. 
	Then for any 
$\beta>\max\{\beta^*_\mathcal{U}(\alpha), \lambda_{k_\alpha+1}(q)\}$ the problem $(GEV;\alpha,\beta)$ has a nontrivial solution $u$ with 
$\E(u)<0$, where $\beta^*_\mathcal{U}(\alpha)$ is defined by \eqref{b_*2}.
\end{theorem}
\begin{proof}
Since $\alpha \in \overline{\mathbb{R}\setminus\sigma(-\Delta_p)}$, we need to investigate two cases: 
\begin{itemize} 
	\item[{\rm (i)}] $\alpha \not\in \sigma(-\Delta_p)$;
	\item[{\rm (ii)}] 
	$\alpha \in \sigma(-\Delta_p)$ and there exists a sequence 
	$\{\alpha_n\}_{n \in \mathbb{N}} \subset \mathbb{R}\setminus\sigma(-\Delta_p)$ such that
	$\lim\limits_{n \to +\infty}\alpha_n=\alpha$.
\end{itemize}

\noindent
\textbf{Case (i).}
Let $\beta > \lambda_{k_\alpha+1}(q)=\max\{\beta^*_\mathcal{U}(\alpha), \lambda_{k_\alpha+1}(q)\}$. Then Lemma~\ref{lem:const-map} guarantees 
the existence of an odd $h_0\in C(S^{k_\alpha},\W)$ and of $t_0>0$ such that
$$
\rho:=\max_{z\in S^{k_\alpha}}\E(t_0 h_0(z)) <0. 
$$
Moreover, by the definition of $k_\alpha$ we have $\alpha < \lambda_{k_\alpha+1}(p)$, and hence Lemma~\ref{lem:bdd-below} implies that 
$\E$ is bounded from below on $Y(\lambda_{k_\alpha+1}(p))$, that is, 
\begin{equation}\label{eq:delta_0}
\delta_0:=\inf\{\E(u):\, u\in Y(\lambda_{k_\alpha+1}(p))\}>-\infty.
\end{equation}
Since $t_0 h_0(\cdot)$ is odd and $\E$ is even,  Lemma~\ref{lem:link} justifies that $\E(t_0 h_0(z_0)) \geq \delta_0$ for some $z_0 \in S^{k_\alpha}_+$, and hence $\delta_0 \leq \rho$.
We are going to show that 
$\E$ has at least one critical value in $[\delta_0-1,\rho]$. 
Suppose, by contradiction, that 
$\E$ has no critical values in $[\delta_0-1,\rho]$. 
Recall that $\E$ satisfies the Palais--Smale condition by Lemma~\ref{PS} 
because we are assuming that $\alpha\not\in\sigma(-\Delta_p)$. Then, due to Lemma~\ref{retract}, there exists 
$\eta\in C([0,1] \times \W,\W)$ such that 
$\eta(s,\cdot)$ is odd for every $s \in [0,1]$ and 
\begin{equation}\label{general:eq:1}
\E(\eta(1,t_0 h_0(z))) \leq \delta_0 - 1
\quad \text{for all} \quad 
z\in S^{k_\alpha}. 
\end{equation}
On the other hand, noting that $\eta(1,t_0 h_0(\cdot))\big|_{S^{k_\alpha}_+} 
\in C(S^{k_\alpha}_+, \W)$ and 
$\eta(1,t_0 h_0(\cdot))\big|_{S^{k_\alpha-1}}$ is odd, 
Lemma~\ref{lem:link} guarantees the existence of 
a point $z_1 \in S^{k_\alpha}_+$ such that 
$\eta(1,t_0 h_0(z_1))\in Y(\lambda_{k_\alpha+1}(p))$, whence 
$\delta_0 \leq \E(\eta(1,t_0 h_0(z_1)))$ 
by the definition of $\delta_0$ (see \eqref{eq:delta_0}).
However, this contradicts \eqref{general:eq:1}. 

\noindent
\textbf{Case (ii).} 
	Let $\beta > \max\{\beta^*_\mathcal{U}(\alpha), \lambda_{k_\alpha+1}(q)\}$. As in the former case, according to Lemma~\ref{lem:const-map}, there exist an odd map
	$h_0\in C(S^{k_\alpha},\W)$ and $t_0 > 0$ such that
	\begin{equation}\label{eq:gen-2-2}
	\rho:=\max_{z\in S^{k_\alpha}} \E(t_0 h_0(z)) <0. 
	\end{equation}
	Recalling that $\alpha < \lambda_{k_\alpha+1}(p)$ and discarding, if necessary, a finite number of terms of the sequence $\{\alpha_n\}_{n \in \mathbb{N}}$, we may suppose that $\alpha_n < \lambda_{k_\alpha+1}(p)$ and  
	\begin{equation}\label{eq:gen-2-3}
	\rho_n:=\max_{z\in S^{k_\alpha}} E_{\alpha_n,\beta}(t_0 h_0(z)) \leq 
	\rho + 
	t_0^p
	\frac{|\alpha_n-\alpha|}{p}
	\max_{z\in S^{k_\alpha}}\|h_0(z)\|_p^p <0
	\end{equation}
	for all $n \in \mathbb{N}$.
	Since $\alpha_n\not\in\sigma(-\Delta_p)$, we apply the proof of the case (i) to each $\alpha_n<\lambda_{k_\alpha+1}(p)$ and 
	$\beta > \lambda_{k_\alpha+1}(q)$, and hence obtain a sequence of critical values $c_n$ of $E_{\alpha_n,\beta}$ such that 
	\begin{equation}\label{eq:gen-2-5}
	\delta_n-1 \leq c_n \leq \rho_n,
	\quad \text{where} \quad 
	\delta_n:=\inf\{E_{\alpha_n,\beta}(u):~ u\in Y(\lambda_{k_\alpha+1}(p))\}>-\infty. 
	\end{equation}
	Let $u_n \in \W$ be a critical point of $E_{\alpha_n,\beta}$ corresponding to the level  
	$c_n$, i.e., 	$E_{\alpha_n,\beta}(u_n)=c_n$.
	We proceed to show that $\{u_n\}_{n \in \mathbb{N}}$ is bounded in $\W$. Suppose, by contradiction, that $\|\nabla u_n\|_p \to +\infty$ as $n \to +\infty$. Set $v_n:=u_n/\|\nabla u_n\|_p$ and note that
	\begin{equation}\label{eq:gen-2-0} 
	\|\E'(u_n)\|_{(\W)^*}=
	\|\E'(u_n)-E_{\alpha_n,\beta}'(u_n)\|_{(\W)^*}
	\leq
	\frac{|\alpha_n-\alpha|}{\lambda_1(p)}\|\nabla u_n\|_p^{p-1}=o(1)\|\nabla u_n\|_p^{p-1}
	\end{equation}
	as $n \to +\infty$. 
	Thus, due to Lemma~\ref{lem:bdd-PS}, we have that $v_n$ converges strongly in $\W$, up to a subsequence, to some $v_0\in ES(p,\alpha) \setminus \{0\}$. Let us prove that $G_\beta(v_0)=0$.
	By \eqref{eq:gen-2-3}, we have 
	\begin{equation}\label{eq:gen-2-6} 
	\left(\frac{1}{q}-\frac{1}{p}\right)G_\beta(v_n)
	=\frac{1}{\|\nabla u_n\|_p^q}\left(E_{\alpha_n,\beta}(u_n)-
	\frac{1}{p}
	\left<
	E_{\alpha_n,\beta}'(u_n),u_n\right> \right) = \frac{c_n}{\|\nabla u_n\|_p^q} \leq 
	\frac{\rho_n}{\|\nabla u_n\|_p^q}<0. 
	\end{equation}
	To obtain a converse estimate, we show that $\delta_n$ is bounded from below. 
	Since $\lim\limits_{n \to +\infty}\alpha_n=\alpha<\lambda_{k_\alpha+1}(p)$, we can choose $\alpha_0$ such that 
	$\alpha_n < \alpha_0 < \lambda_{k_\alpha+1}(p)$ for all sufficiently large $n \in \mathbb{N}$. 
	Thus,  Lemma~\ref{lem:bdd-below} implies that 
	$E_{\alpha_0,\beta}$ is bounded from below on $Y(\lambda_{k_\alpha+1}(p))$. 
	Noting that $E_{\alpha_n,\beta}(u) \geq E_{\alpha_0,\beta}(u)$ 
	for any $u \in \W$, we get 
	$\delta_n \geq \inf\{E_{\alpha_0,\beta}(u):~ u\in Y(\lambda_{k_\alpha+1}(p))\} > -\infty$ 
	for all $n \in \mathbb{N}$ large enough, which is the desired boundedness. 
	Using this fact, the two equalities in \eqref{eq:gen-2-6}, and \eqref{eq:gen-2-5}, we derive that
	\begin{equation}\label{eq:gen-2-7} 
	0>\left(\frac{1}{q}-\frac{1}{p}\right)G_\beta(v_n)
	\geq  
	\frac{\delta_n-1}{\|\nabla u_n\|_p^q} \to 0
	\end{equation}
	as $n \to +\infty$, which leads to
	$G_\beta(v_0)=0$, because $v_n \to v_0$ strongly in $\W$. On the other hand, since $\alpha \in \sigma(-\Delta_p)$ and $\beta > \beta^*_\mathcal{U}(\alpha)$, we get
	\begin{equation}\label{eq:gen-2-1}
	G_\beta(\varphi)=\|\nabla \varphi\|_q^q-\beta\|\varphi\|_q^q \neq 0 
	~\text{ for all }~
	\varphi\in ES(p;\alpha)\setminus\{0\},
	\end{equation}
	see Lemma~\ref{lem:G-negative}. Hence, we obtain a contradiction, since $G_\beta(v_0) = 0$ and $v_0\in ES(p,\alpha) \setminus \{0\}$. 
	Thus, from  \eqref{eq:gen-2-0} it follows that $\{u_n\}_{n \in \mathbb{N}}$ is a 
	\textit{bounded} Palais--Smale sequence for $\E$. 
	Then, the $(S_+)$-property of the operator $-\Delta_p - \Delta_q$ (see Remark~\ref{remark:S_+}) implies that
	$u_n$ converges strongly in $\W$, up to a subsequence, 
	to some critical point $u_0$ of $\E$. 
	Furthermore, $u_0$ is nontrivial and its energy is negative, since 
	$$
	\E(u_0)=\limsup_{n \to +\infty} E_{\alpha_n,\beta}(u_n) 
	= \limsup_{n \to +\infty} c_n \leq \limsup_{n \to +\infty} \rho_n 
	\leq \rho + o(1) < 0
	$$
	by \eqref{eq:gen-2-2} and \eqref{eq:gen-2-3}. 
\end{proof}

\begin{remark}\label{rem:negarive}
	Note that the proof of the case (ii) gives more. Namely, if $\alpha \in \sigma(-\Delta_p)$ and $\lim\limits_{n \to +\infty}\alpha_n=\alpha$ for some  sequence 
	$\{\alpha_n\}_{n \in \mathbb{N}} \subset \mathbb{R}\setminus\sigma(-\Delta_p)$, and $\beta > \lambda_{k_\alpha+1}(q)$ is such that \eqref{eq:gen-2-1} holds, then there exists a nontrivial solution to $(GEV;\alpha, \beta)$.
\end{remark}

\subsection{General existence result via the descending flow}
In the last part of this section, we use the descending flow method to provide an existence result for $(p,q)$-Laplace equations with a nonlinearity in the general form.

Suppose that 
$h\colon \Omega \times \mathbb{R} \to \mathbb{R}$ is 
a Carath\'eodory function satisfying $h(x,0)=0$ for a.e.\ $x \in \Omega$ 
and there exists $C>0$ such that 
\begin{equation}\label{h}
|h(x,s)| \leq C(1 + |s|^{p-1}) 
~\text{ for every }~
s \in \mathbb{R}
~\text{ and a.e.\ }~ x \in \Omega.
\end{equation}
Under \eqref{h}, we define a $C^1$-functional $J$ on $\W$ by 
\begin{equation}\label{eq:J}
J(u) := \frac{1}{p}\intO|\nabla u|^p\,dx 
+\frac{1}{q}\intO |\nabla u|^q\,dx -\intO \int_0^{u(x)} h(x,s)\,ds \, dx.
\end{equation}
For simplicity, we denote the positive cone in 
$C^1_0(\overline{\Omega})$ 
by
\begin{equation}\label{eq:P}
P:=\{u\in C_0^1(\overline{\Omega}):~ u(x)> 0 ~\text{ for all }~ x \in \Omega\}. 
\end{equation}

The following result can be proved by the same arguments as \cite[Theorem~11]{MT2}. 
For the reader's convenience, 
we give a sketch of the proof in Appendix~B. 

\begin{theorem}\label{thm:1} 
Assume that the following conditions hold: 
\begin{itemize} 
\item[$(A1)$] there exists $\lambda_0>0$ such that 
$$
h(x,u)\,u + \lambda_0 (|u|^q+|u|^p) \geq 0 
~\text{ for every }~ u \in \mathbb{R} 
~\text{ and a.e.\ }~ x \in \Omega; 
$$
\item[$(A2)$] there exists $\gamma\in C([0,1], \C)$ such that 
$\gamma(0)\in P$, $\gamma(1)\in -P$ and 
$\max\limits_{s \in [0,1]} J(\gamma(s))<0$. 
\end{itemize} 
If, moreover,  $J$ is coercive on $\W$, then $J$ has at least three critical points 
$w_1\in {\rm int\,}P$, $w_2 \in -{\rm int\,}P$,  
and $w_3\in \C \setminus (P\cup -P)$,  such that 
$J(w_i) \leq \max\limits_{s \in [0,1]} J(\gamma(s))<0$ for $i=1, 2, 3$. Here
$$
{\rm int\,}P := \{u \in P:~ \partial u(x)/\partial \nu < 0 ~\text{ for all }~ x \in \partial \Omega\},
$$
and $\nu$ denotes the unit outer normal vector to $\partial\Omega$.
\end{theorem} 

We say that $v \in \W$ is a (weak) super-solution of $(GEV;\alpha,\beta)$ whenever for all nonnegative $\varphi \in \W$ there holds
$$
\intO |\nabla v|^{p-2}\nabla v\nabla \varphi\,dx
+\intO |\nabla v|^{q-2}\nabla v\nabla \varphi\,dx
\geq \alpha\intO |v|^{p-2}v\varphi\,dx+\beta\intO |v|^{q-2}v\varphi\,dx.
$$
Applying Theorem~\ref{thm:1} to a truncated functional corresponding to $E_{\alpha, \beta}$, 
we show the following result on the existence of nodal solutions 
to $(GEV; \alpha, \beta)$ with a negative energy. 
\begin{proposition}\label{prop:1} 
Let $\alpha \in \mathbb{R}$ and  $\beta>\lambda_2(q)$. 
If there exists a super-solution of $(GEV;\alpha,\beta)$ which belongs to ${\rm int}\, P$, then $(GEV;\alpha,\beta)$ has a nodal solution $u$ such that $\E(u)<0$. 
\end{proposition}
\begin{proof}
Let $v\in {\rm int}\, P$ be a super-solution of $(GEV;\alpha,\beta)$ 
with $\alpha\in\mathbb{R}$ and $\beta>\lambda_2(q)$. 
Note that $-v$ becomes a negative sub-solution of $(GEV;\alpha,\beta)$.
Using $v$, we truncate the right-hand side of $(GEV;\alpha,\beta)$ as follows: 
\begin{equation*} 
f(x, s):=
\begin{cases} 
\alpha v(x)^{p-1} +\beta v(x)^{q-1}  &\text{ if } s> v(x), \\
\alpha |s|^{p-2}s + \beta|s|^{q-2}s
&\text{ if } -v(x) \leq s \leq  v(x), \\
-\alpha v(x)^{p-1} -\beta v(x)^{q-1}  
&\text{ if } s< -v(x). 
\end{cases} 
\end{equation*}
It is easy to see that 
$f$ is the Carath\'eodory function and $f(x,0)=0$ for all $x \in \Omega$. Moreover, 
$f$ satisfies \eqref{h} and, taking $\lambda_0 = \max\{|\alpha|,|\beta|\}$, it satisfies the assumption $(A1)$ of Theorem~\ref{thm:1}.

Define a corresponding \textit{truncated} $C^1$-functional $I$ on $\W$ by 
\begin{align*}
I(u) := 
\frac{1}{p}\intO|\nabla u|^p\,dx 
+\frac{1}{q}\intO |\nabla u|^q\,dx 
-\intO \int_0^{u(x)} f(x,s)\,ds\,dx.
\end{align*}
Note that the boundedness of $v$ in $\Omega$ implies the boundedness of $f$, and therefore $I$ is coercive on $\W$. To apply Theorem~\ref{thm:1} it remains to show that $(A2)$ holds. 
To this end, let us construct an appropriate path $\gamma_0$. 
Choose $\varepsilon>0$ satisfying $\lambda_2(q)+2\varepsilon <\beta$. 
By the characterization \eqref{second:mp} of $\lambda_2(q)$, there exists 
$\gamma\in C([0,1],S(q))$ such that 
$\gamma(0)=\varphi_q\in {\rm int}\,P$, 
$\gamma(1)=-\varphi_q\in -{\rm int}\,P$, and 
$\max\limits_{s \in [0,1]}\|\nabla \gamma(s)\|_q^q < \lambda_2(q)+\varepsilon$. 
Using the density arguments (as in the proof of Lemma~\ref{lem:const-map}), we can obtain a path
$\tilde{\gamma}\in C([0,1], \C\setminus\{0\})$ such that 
$\tilde{\gamma}(0)\in P$, $\tilde{\gamma}(1)\in -P$, and 
$$
\|\nabla \tilde{\gamma}(s)\|_q^q \leq 
(\lambda_2(q)+2\varepsilon)\|\tilde{\gamma}(s)\|_q^q 
$$
for every $s \in [0,1]$. 
Since $v\in {\rm int}\,P$ and $\tilde{\gamma}\in C([0,1], \C\setminus\{0\})$, we get for any $t > 0$ small enough, $s \in [0,1]$ and $x \in \Omega$ that
$$
-v(x) \leq t \tilde{\gamma}(s)(x) \leq v(x), 
$$
and hence
 $f(x,t \tilde{\gamma}(s))=
t^{p-1}\alpha |\tilde{\gamma}(s)|^{p-2}\tilde{\gamma}(s)
+t^{q-1}\beta|\tilde{\gamma}(s)|^{q-2}\tilde{\gamma}(s)$.
Therefore,
\begin{align*}
I(t \tilde{\gamma}(s)) 
&= \frac{t^{p}}{p}
\left(\|\nabla \tilde{\gamma}(s)\|_p^p - \alpha\|\tilde{\gamma}(s)\|_p^p
\right)
+
\frac{t^{q}}{q}
\left(\|\nabla \tilde{\gamma}(s)\|_q^q-\beta\|\tilde{\gamma}(s)\|_q^q \right) 
\\
&\leq 
t^{q}
\left(
\frac{t^{p-q}}{p}
\left[
\max_{s \in [0,1]}\|\nabla \tilde{\gamma}(s)\|_p^p 
+ |\alpha|\max_{s \in [0,1]}\|\tilde{\gamma}(s)\|_p^p
\right] + 
\frac{\lambda_2(q)+2\varepsilon-\beta}{q} 
\min_{s \in [0,1]}\|\tilde{\gamma}(s)\|_q^q 
\right)<0
\end{align*}
for sufficiently small $t>0$, 
since $q<p$, $\min\limits_{s \in [0,1]}\|\tilde{\gamma}(s)\|_q^q > 0$, and $\lambda_2(q)+2\varepsilon <\beta$. Thus, for such a small $t > 0$ the path $\gamma_0(s) := t \tilde{\gamma}(s)$  satisfies the assumption $(A2)$ of Theorem~\ref{thm:1}.

As a result, according to Theorem~\ref{thm:1}, 
we obtain a sign-changing critical point $u\in \C \setminus (P\cup -P)$ of $I$ 
satisfying $I(u) \leq \max\limits_{t\in[0,1]} I(\gamma_0(t))<0$. 
By the standard argument, we can show that $-v\leq u\leq v$ in $\Omega$. 
In fact, recalling that $v$ is a super-solution of $(GEV;\alpha,\beta)$ and taking $(u-v)^+\in\W$ as a test function for $I'(u) - E_{\alpha,\beta}'(v)$, we obtain 
\begin{align*}
0 
&\leq \int_{u>v} 
\left(|\nabla u|^{p-2}\nabla u -|\nabla v|^{p-2}\nabla v\right)
\left(\nabla u-\nabla v\right)\,dx 
\\ 
&\qquad +\int_{u>v} 
\left(|\nabla u|^{q-2}\nabla u -|\nabla v|^{q-2}\nabla v\right)
\left(\nabla u-\nabla v\right)\,dx 
\\
&\le \intO f(x,u)(u-v)^+\,dx
-\intO (\alpha v^{p-1}+\beta v^{q-1})(u-v)^+\,dx =0,
\end{align*}
which implies that $(u-v)^+\equiv 0$ and hence $u \leq v$ in $\Omega$. 
Similarly, taking $-(u-(-v))^-$ as a test function, 
we get $u \geq -v$. 
Therefore, $u$ is a nodal solution of $(GEV;\alpha,\beta)$ 
and $\E(u) = I(u) \leq \max\limits_{s \in [0,1]} I(\gamma_0(s)) < 0$. 
\end{proof}

\section{Proofs of the main results}\label{sec:proofs}
In this section, we collect the proofs of our main results stated in Subsection~\ref{subsec:main}. 

\begin{proof*}{Theorem~\ref{thm:nonexist}} 
Recall that any sign-changing solution of $(GEV;\alpha,\beta)$ belongs to 
the nodal Nehari set $\mathcal{M}_{\alpha,\beta}$ defined by \eqref{def:M}. 
At the same time,  $\mathcal{M}_{\alpha,\beta}$ is empty under the assumptions of the theorem, 
as is shown in Lemma~\ref{lemma:emptinessM}, which completes the proof.
\end{proof*}

\begin{proof*}{Theorem~\ref{thm:positive}} 
The desired conclusion follows directly from the combination of Theorem~\ref{thm:minimizer} and 
Lemmas~\ref{lem:solution} and \ref{lem:twonodal}.
\end{proof*}

\begin{proof*}{Theorem~\ref{thm:neg1}}
	Note that problem	 $(GEV;\alpha,\beta)$ possesses an abstract nontrivial solution $u \in \W$ with $\E(u) < 0$ for any $\alpha \in \overline{\mathbb{R}\setminus\sigma(-\Delta_p)}$ and
	$\beta > \max\left\{\beta^{*}_{\mathcal{U}}(\alpha), \lambda_{k_\alpha+1}(q)\right\} \geq \lambda_2(q)$ by Theorem~\ref{thm:general}. 
	If $u$ is a nodal solution, then we are done.
	If $u$ is a nontrivial nonnegative solution, then 
	 $u \in {\rm int}\,P$ 
	(see, e.g., \cite[Remark~1, p.~3284]{BobkovTanaka2015}), and hence Proposition~\ref{prop:1} 
	guarantees the existence of 
	a nodal solution $v$ of $(GEV;\alpha,\beta)$ 
	such that $\E(v)<0$.
\end{proof*} 

\begin{proof*}{Theorem~\ref{thm:neg2}}
	If $\alpha<\lambda_1(p)$ or $\lambda_1(p) < \alpha < \lambda_2(p)$, then for all $\beta > \lambda_2(q)$ there exists a nodal solution, as
	follows from Theorem~\ref{thm:neg1}. 
	If $\alpha = \lambda_1(p)$, then, as noted in Remark~\ref{rem:negarive}, Theorem~\ref{thm:general} implies the existence of an abstract nontrivial negative energy solution of $(GEV;\alpha,\beta)$ for any $\beta > \lambda_2(q)$ such that $G_\beta(\varphi_p) \neq 0$. Since the first eigenfunction $\varphi_p$ of $-\Delta_p$ is unique, up to a multiplier, we derive the existence under the assumption $\beta\not=\|\nabla \varphi_p\|_q^q/\|\varphi_p\|_q^q$.
	If the obtained solution changes its sign, then we are done. Otherwise, we apply Proposition~\ref{prop:1} and obtain the existence of a nodal solution with a negative energy.
\end{proof*}

Finally, we will prove the nonexistence result in the one-dimensional case. 
\begin{proof*}{Theorem~\ref{nonexist:N1}} 
Let $N=1$ and $\Omega = (0,T)$, $T>0$.
We temporarily denote by $\lambda_k(r,S)$ the $k$th eigenvalue of $-\Delta_r$ on $(0,S)$ subject to zero Dirichlet boundary conditions, $r>1$, $S>0$ (see Appendix~A). 
Suppose, by contradiction, that $\alpha\leq \lambda_2(p,T)$ and $\beta\leq \lambda_2(q,T)$, but there exists a nodal solution $u$ for 
$(GEV;\alpha,\beta)$. 
Evidently, there is at least one nodal domain of $u$ which length $S$ is less than or equal to $T/2$. 
Using, if necessary, the translation of the coordinate axis, we may assume that 
$u$ is a constant-sign solution of $(GEV;\alpha, \beta)$ on interval $(0,S)$.
Define $v := u$ on $(0,S)$ and $v = 0$ on $[S, T/2]$. Clearly, $v \in W_0^{1,p}(0,S) \subset W_0^{1,p}(0,T/2)$. 
Moreover, it is not hard to see that
\begin{equation*}
\lambda_2(r,T)=\lambda_1(r,T/2)=
\left(\frac{2S}{T}\right)^r\lambda_1(r,S) \leq \lambda_1(r,S) 
\end{equation*}
for any $r>1$. 
Thus, \eqref{charact-1st-ev} and the assumption $S \leq T/2$ lead to the inequalities
\begin{align} 
\label{eq:N1-3}
\alpha \leq \lambda_2(p,T) \leq \lambda_1(p,S) 
\leq \frac{\int_0^S |v'|^p\,dt}{\int_0^S |v|^p\,dt}
\quad \text{and} \quad 
\beta \leq \lambda_2(q,T) \leq \lambda_1(q,S) 
\leq \frac{\int_0^S |v'|^q\,dt}{\int_0^S |v|^q\,dt}. 
\end{align}
Taking now $v$ as a test function for \eqref{weaksolution}, 
we arrive at
$$
0\leq \int_0^S |v'|^p\,dt-\alpha\int_0^S |v|^p\,dt
=\beta\int_0^S |v|^q\,dt - \int_0^S |v'|^q\,dt \leq 0, 
$$
and hence we have equalities in 
\eqref{eq:N1-3}. 
On the other hand, the simplicity of $\lambda_1(r,S)$ implies that $v$ is the first eigenfunction corresponding to 
$\lambda_1(p,S)$ and $\lambda_1(q,S)$,  simultaneously. 
However, this is a contradiction, since $\varphi_p$ and $\varphi_q$ are linearly independent for $N=1$ (see \cite[Lemma~4.3]{KTT} or Lemma~\ref{IV} below).
\end{proof*}

\par
\bigskip
\noindent
{\bf Acknowledgments.} This work was 
supported by JSPS KAKENHI Grant Number 15K17577. 
The first author wishes to thank Tokyo University of Science, where the main constructions and results of the article were obtained, for the invitation and hospitality. 
The work of the first author was also supported by the project LO1506 of the Czech Ministry of Education, Youth and Sports.

\section*{Appendix~A} 
\addcontentsline{toc}{section}{\protect\numberline{}Appendix~A}

In this section, we show some relations between eigenvalues and eigenfunctions of the $p$- and $q$-Laplacians in the one-dimensional case. Consider the eigenvalue problem
$$
\left\{
\begin{aligned}
-(|u'|^{r-2} u')' &= \lambda |u|^{r-2} u \quad {\rm in}\ (0,T), \\
u(0) = u(T) &= 0, 
\end{aligned}
\right.
$$
where $r > 1$ and $T > 0$.
It is known (cf. \cite[Theorem~3.1]{drabman}) that $\sigma(-\Delta_r)$ is exhausted by eigenvalues $\lambda_k(r) =  (r-1)\left(\frac{k \pi_r}{T}\right)^r$,
where $\pi_r = \frac{2\pi}{r \sin(\pi/r)}$. (It is not hard to see that $\pi_r$ is a decreasing function of $r>1$.) The corresponding eigenfunctions are denoted by $\sin_r\left(\frac{k \pi_r t}{T}\right)$, where $\sin_r(t)$ is the inverse function of $\int_0^x (1-s^r)^{-1/r} \, ds$, $x \in [0,1]$, extended periodically and anti-periodically from $[0, \pi_r/2]$ to the whole  $\mathbb{R}$ (see also \cite{busheled}). 
By construction, $\sin_r\left(\frac{k \pi_r t}{T}\right)$ has exactly $k$ nodal domains of the length $T/k$ on $(0, T)$.
As usual, we denote the first eigenfunction $\sin_r\left(\frac{\pi_r t}{T}\right)$ as $\varphi_r$.

For the convenience of the reader we briefly prove that the first eigenfunctions $\varphi_p$ and $\varphi_q$ are linearly independent; see also \cite[Lemma 4.3]{KTT} for a different proof. 

\begin{lemma}\label{IV} 
Let $N=1$ and $q \neq p$. 
Then $\varphi_p$ and $\varphi_q$ are linearly independent. 
\end{lemma} 
\begin{proof} 
	Suppose, by contradiction, that $\varphi_p(t) = \varphi_q(t)$ for all $t \in [0,T]$. In particular, we have $$
	\sin_p\left(\frac{\pi_p t}{T}\right) = \sin_q\left(\frac{\pi_q t}{T}\right)
	$$ 
	for all $t \in [0, T/2]$. By the definitions of $\sin_p$ and $\sin_q$, we obtain
	$$
	\frac{1}{\pi_p} \int_0^x (1-s^p)^{-1/p} \, ds = 
	\frac{1}{\pi_q} \int_0^x (1-s^q)^{-1/q} \, ds ~\text{ for all }~ x \in [0,1].
	$$
	Using a Taylor series, we get $(1-s^p)^{-1/p} = 1 + O(s^p)$ and 
	$(1-s^q)^{-1/q} = 1 + O(s^q)$ in a neighborhood of $s = 0$.
	Thus, 
	$$
	\int_0^x \left[\left(\frac{1}{\pi_p} - \frac{1}{\pi_q} \right)  + O(s^p) + O(s^q) \right] \, ds = 0
	$$	
	for sufficiently small $x > 0$, which implies that $\pi_p = \pi_q$, since $p,q>1$. However, this contradicts the monotonicity of $\pi_r$ with respect to $r>1$.
\end{proof}

Next, we prove the main result of the section.
\begin{lemma}\label{lem:appendixA}
Let $N=1$ and $1<q<p<+\infty$. 
Then $\lambda_1(q) < \frac{\|\varphi_p'\|^q_q}{\|\varphi_p\|^q_q} < \lambda_2(q)$.
\end{lemma}
\begin{proof}
	The first inequality is trivial because the first eigenvalue $\lambda_1(q)$ is simple and $\varphi_p \neq \varphi_q$ 
	(see \cite{KTT} or Lemma~\ref{IV}). 
	Let us prove by direct calculations that 
	$$
	\frac{\|\varphi_p'\|^q_q}{\|\varphi_p\|^q_q} < \lambda_2(q)
	$$ 
	for $q < p$.
	Note that
	\begin{equation}\label{eq:beta0}
	\frac{\|\varphi_p'\|^q_q}{\|\varphi_p\|^q_q} = 
	\frac{ \int\limits_{0}^T \left|\sin_p' \left(\frac{\pi_p t}{T}\right) \right|^q \, dt}{\int\limits_{0}^T \left|\sin_p \left(\frac{\pi_p t}{T}\right) \right|^q \, dt} = 
	\frac{\pi_p^q}{T^q}
	\frac{ \int\limits_{0}^T \left|\cos_p \left(\frac{\pi_p t}{T}\right) \right|^q \, dt}{\int\limits_{0}^T \left|\sin_p \left(\frac{\pi_p t}{T}\right) \right|^q \, dt} = 
	\frac{\pi_p^q}{T^q}
	\frac{ \int\limits_{0}^{\pi_p} 
		\left|\cos_p x \right|^q \, dx}{\int\limits_{0}^{\pi_p} \left|\sin_p x \right|^q \, dx} 
	=
	\frac{\pi_p^q}{T^q}
	\frac{ \int\limits_{0}^{\pi_p/2} 
		\cos_p^q x \, dx}{\int\limits_{0}^{\pi_p/2} \sin_p^q x \, dx}.
	\end{equation}
	Using the formulas 
	$$
	\int_{0}^{\pi_p/2} \sin_p^q x \, dx = \frac{1}{p} B\left(\frac{q+1}{p}, \frac{p-1}{p}\right) 
	\quad \text{and} \quad 
	\int_{0}^{\pi_p/2} 
	\cos_p^q x \, dx = 
	\frac{1}{p} B\left(\frac{1}{p}, 1+ \frac{q-1}{p}\right)
	$$
	from \cite[Proposition~3.1]{busheled}, 
	where $B(x,y) := \int_0^1 t^{x-1} (1-t)^{y-1} \, dt$ is the beta function with real $x, y > 0$,  it becomes sufficient to prove that
	\begin{equation}\label{eq:beta1}
	\frac{B\left(\frac{1}{p}, 1+ \frac{q-1}{p}\right)}{B\left(\frac{q+1}{p}, \frac{p-1}{p}\right)}
	< \frac{\lambda_2(q) T^q}{\pi_p^q} \equiv  
	(q-1)\left(\frac{2 \pi_q}{\pi_p}\right)^q.
	\end{equation}
	We will subsequently simplify \eqref{eq:beta1}, to obtain an easier sufficient condition.
	Note that, by definition, 
	$$
	B\left(\frac{1}{p}, 1+ \frac{q-1}{p}\right) = 
	\int_0^1 t^{\frac{1}{p}-1} (1-t)^{\frac{q-1}{p}} \, dt < 
	\int_0^1 t^{\frac{1}{p}-1} \, dt = 
	B\left(\frac{1}{p}, 1\right) = p.
	$$
	Note that $B(x,y) = \frac{\Gamma(x) \Gamma(y)}{\Gamma(x+y)}$, where $\Gamma(y)$ is the gamma function, cf. \cite[Theorem~1.1.4]{andrews}.
	Hence, combining the Euler reflection formula $\Gamma(y) \Gamma(1-y) = \frac{\pi}{\sin \pi y}$ (see, e.g., \cite[p.~9]{andrews}) with the identity $x \Gamma(x) = \Gamma(x+1)$, we obtain
	\begin{equation}\label{beta:ident}
	B(x,y) \cdot B(x+y, 1-y) = \frac{\Gamma(x) \Gamma(y)}{\Gamma(x+y)} \cdot \frac{\Gamma(x+y) \Gamma(1-y)}{\Gamma(x+1)} = \frac{\Gamma(x)}{\Gamma(x+1)}\, \Gamma(y) \Gamma(1-y) = \frac{\pi}{x \sin \pi y}.
	\end{equation}
	Applying \eqref{beta:ident} to 
	$B\left(\frac{q+1}{p}, \frac{p-1}{p}\right)$ with $x=q/p$ and $y=1/p$, we get
	$$
	B\left(\frac{q+1}{p}, \frac{p-1}{p}\right) = 
	\frac{p \pi}{q \sin\left(\frac{\pi}{p}\right)}
	\cdot
	\frac{1}{B\left(\frac{q}{p}, \frac{1}{p}\right)}.
	$$
	Therefore, using the estimate
	$$
	B\left(\frac{q}{p}, \frac{1}{p}\right)=
	\int_0^1 t^{\frac{q}{p}-1} (1-t)^{\frac{1}{p}-1} \, dt < 
	\int_0^1 t^{\frac{1}{p}-1} (1-t)^{\frac{1}{p}-1} \, dt = 
	B\left(\frac{1}{p}, \frac{1}{p} \right),
	$$
	we arrive at
	\begin{equation}\label{eq:beta2}
	\frac{B\left(\frac{1}{p}, 1+ \frac{q-1}{p}\right)}{B\left(\frac{q+1}{p}, \frac{p-1}{p}\right)} <
	\frac{q \sin\left(\frac{\pi}{p}\right)}{\pi}
	B\left(\frac{1}{p}, \frac{1}{p} \right) = 
	\frac{2 q}{p \pi_p} B\left(\frac{1}{p}, \frac{1}{p} \right).
	\end{equation}
	
	Thus, comparing the right-hand sides of \eqref{eq:beta1} and  \eqref{eq:beta2}, we get the following sufficient condition for the assertion of the lemma:
	\begin{equation}\label{eq:beta3}
	 \frac{1}{p}B\left(\frac{1}{p}, \frac{1}{p} \right)\leq 2^{q-1} \frac{q-1}{q} \frac{\pi_q^q}{\pi_p^{q-1}}.
	\end{equation}
	To prove this inequality, we first obtain an appropriate upper bound for its left-hand side. From \cite[p.~8]{andrews} we know that 
	\begin{equation}\label{eq:beta5}
	\frac{1}{p}B\left(\frac{1}{p}, \frac{1}{p} \right) = 
	\frac{1}{p}\, 2 p \prod_{n=1}^{\infty}\frac{1+\frac{2}{n p}}{\left(1+\frac{1}{n p}\right)^2} 
	<
	2 \frac{1+\frac{2}{p}}{\left(1+\frac{1}{p}\right)^2} 
	= \frac{2 p (p + 2)}{(p+1)^2},
	\end{equation}
	since for all $n \in \mathbb{N}$ there holds
	$$
	\frac{1+\frac{2}{np}}{\left(1+\frac{1}{np}\right)^2} = \frac{1+\frac{2}{np}}{1+\frac{2}{np}+\left(\frac{1}{np}\right)^2} < 1.
	$$ 
	Next, we will get a suitable lower bound for the right-hand side of \eqref{eq:beta3}. 
	Since $\pi_r$ is a decreasing function of $r>1$ (in fact $d\pi_r/dr<0$), we have $\pi_q/\pi_p > 1$ for $q<p$. Hence, 
	\begin{equation}\label{eq:beta6}
	2^{q-1} \frac{q-1}{q} \frac{\pi_q^q}{\pi_p^{q-1}} > 
	2^{q-1} \frac{q-1}{q} \pi_q = 
	\frac{2^{q} \pi (q-1)}{q^2 \sin\left(\frac{\pi}{q}\right)} = 
	\frac{2^{q}}{q} \cdot
	\frac{\frac{\pi}{q}(q-1)}{ \sin\left(\frac{\pi}{q}(q-1)\right)} > \frac{2^{q}}{q},
	\end{equation}
	since $\sin x < x$ for all $x > 0$. 
	
	Let us consider three cases. Assume first that $1<q<p\leq 2$. By a direct analysis, 
	the minimum value of the right-hand side $2^q/q$ of \eqref{eq:beta6} is greater than $16/9$. 
	Since the right-hand side 
	$$
	\frac{2 p (p + 2)}{(p+1)^2}
	$$ 
	of \eqref{eq:beta5} is strictly increasing with respect to $p>1$, it is easy to see that 
	$$
	\frac{2 p (p + 2)}{(p+1)^2} \leq \frac{16}{9} \quad \text{for all} \quad 1<p\leq 2.
	$$
	Combining these facts, we prove that \eqref{eq:beta3} holds for $1 < q < p \leq 2$.
	
	Secondly, assume that $2\leq q<p$. 
	Noting that $2^q/q$ is, in fact, strictly increasing for $q \geq 2$, we obtain
	$$
	\frac{2^q}{q} \geq \left.\frac{2^r}{r}\right|_{r=2} = 2 
	>  \frac{2 p (p + 2)}{(p+1)^2} = 2 \, \frac{p^2+2p}{p^2+2p+1}
	$$ 
	for all $q \geq 2$ and $p>1$. Thus, \eqref{eq:beta5} and \eqref{eq:beta6} yield 	\eqref{eq:beta3} for $2\leq q<p$. 
	
	Finally we assume that $1 < q < 2 \leq p$. Since $\pi_r$ is decreasing, $p \geq 2$ implies that $\pi_p \leq \pi$, and we refine inequality \eqref{eq:beta6} in the following way:
	\begin{equation*}
	2^{q-1} \frac{q-1}{q} \frac{\pi_q^q}{\pi_p^{q-1}} \geq  
	\frac{2^{q-1}}{\pi^{q-1}} \frac{q-1}{q} \frac{2^q \pi^q}{q^q \sin^q\left(\frac{\pi}{q}\right)} \geq 
	\frac{2^{2q-1}}{q^q} \frac{q-1}{q} \frac{\pi}{\sin\left(\frac{\pi}{q}\right)} = 
	\frac{2^{2q-1}}{q^q}
	\cdot
	\frac{\frac{\pi}{q}(q-1)}{ \sin\left(\frac{\pi}{q}(q-1)\right)} > \frac{2^{2q-1}}{q^q}.
	\end{equation*}		
	It is not hard to check that 
	$$
	\frac{2^{2q-1}}{q^q} > 2 > \frac{2 p (p + 2)}{(p+1)^2} \quad  \text{for all} \quad q \in (1,2),
	$$ 
	which again implies \eqref{eq:beta3}.
	
	Therefore, \eqref{eq:beta3} holds for all $1<q<p<+\infty$, which completes the proof.
\end{proof}
	
	If we swap $p$ and $q$ in Lemma~\ref{lem:appendixA}, then an opposite situation occurs.
\begin{lemma} 
	Let $N=1$. Then for any $k \in \mathbb{N}$ there exist $1 < q_0 < p_0$, such that $\frac{\|\varphi_q^\prime\|^p_p}{\|\varphi_q\|^p_p} > \lambda_k(p)$ for all $1 < q < q_0$ and  $p > p_0$.
\end{lemma}
\begin{proof}
	The case $k=1$ is obvious. 
	Let $k \geq 2$. Similar to \eqref{eq:beta0} and \eqref{eq:beta1}, it is sufficient to show that
	\begin{equation}\label{eq:gamma1}
	\frac{B\left(\frac{1}{q}, 1+ \frac{p-1}{q}\right)}{B\left(\frac{p+1}{q}, \frac{q-1}{q}\right)}
	>
	(p-1)\left(\frac{k \pi_p}{\pi_q}\right)^p.
	\end{equation}
	Note first that 
	$$
	B\left(\frac{1}{q}, 1+ \frac{p-1}{q}\right) = 
	\int_0^1 t^{\frac{1}{q}-1} (1-t)^{\frac{p-1}{q}} \, dt 
	>
	\int_0^1 (1-t)^{\frac{p-1}{q}} \, dt = 
	B\left(1, 1+ \frac{p-1}{q}\right) = \frac{q}{p+q-1},
	$$
	and for $q<p$ there holds
	$$
	B\left(\frac{p+1}{q}, \frac{q-1}{q}\right) = 
	\int_0^1 t^{\frac{p+1}{q}-1} (1-t)^{-\frac{1}{q}} \, dt 
	<
	\int_0^1 (1-t)^{-\frac{1}{q}} \, dt = 
	B\left(1, \frac{q-1}{q}\right) = \frac{q}{q-1}.
	$$
	Therefore, \eqref{eq:gamma1} can be simplified as
	$$
	\left(\frac{q-1}{p+q-1}\right)^{\frac{1}{p}} 
	> \frac{kq(p-1)^\frac{1}{p}}{p} \,  \frac{\sin\left(\frac{\pi}{q}\right)}{\sin\left(\frac{\pi}{p}\right)}.
	$$
	Note that $\sin\left(\frac{\pi}{q}\right) = \sin\left(\frac{\pi}{q}(q-1)\right) < \frac{\pi}{q}(q-1)$. Hence, using the estimates
	$(p+q-1)^{1/p} < 2 p^{1/p}$ and
	$(p-1)^{1/p} < p^{1/p}$, we arrive at the following sufficient inequality:
	$$
	\sin\left(\frac{\pi}{p}\right) > 
	\frac{2 \pi k (q-1)^{\frac{p-1}{p}}}{p^{\frac{p-2}{p}}} = 
	p^{\frac{2}{p}} \cdot \frac{2\pi k(q-1)^{\frac{p-1}{p}}}{p}. 
	$$
	At the same time,  $(q-1)^{\frac{p-1}{p}} \leq (q-1)^{\frac{1}{2}}$ for $1<q<2<p$, and, choosing $p_1>2$ large enough, we obtain $2 \geq p^{2/p}$ for any $p \geq p_1$. Therefore, to prove \eqref{eq:gamma1}
	it is sufficient to show that
	\begin{equation}\label{eq:betal}
	\frac{4 \pi k (q-1)^{\frac{1}{2}}}{p} <
	\sin\left(\frac{\pi}{p}\right) = \frac{\pi}{p} + o\left(\frac{\pi}{p}\right).
	\end{equation}
	However, \eqref{eq:betal} is obviously satisfied for any $q < 1+ \frac{1}{(4k)^2}$ and sufficiently large $p > p_1$.
\end{proof}

\section*{Appendix~B: Sketch of the proof of Theorem~\ref{thm:1}}\label{B}
\addcontentsline{toc}{section}{\protect\numberline{}Appendix~B}

Let us consider a map $T_\lambda:\W\to (\W)^*$ defined for $\lambda>0$ by
\begin{equation*}
	\left<T_\lambda(u),v\right>=
	\int_\Omega \left(|\nabla u|^{p-2}+|\nabla u|^{q-2}\right)\nabla u \nabla v \, dx
	+ \lambda \int_\Omega \left(|u|^{p-2}+|u|^{q-2}\right) u v\, dx
\end{equation*}
for $u, v \in \W$. 
The following properties of $T_\lambda$ can be proved in much the same way as in the proof of \cite[Propositions~9, 10]{MT2}.
\begin{lemma}\label{T} 
	$T_\lambda$ is invertible and  $T_\lambda^{-1}\colon (\W)^* \to \W$ is continuous. 
	Moreover, if $1<p\le N$ and $r>N/p$, then there exists a constant 
	$D_0>0$ such that  for all $u \in L^r(\Omega)$ we have
	$$
	\|T_\lambda^{-1}(u)\|_\infty \leq D_0\|u\|_r^{1/(p-1)}.
	$$
\end{lemma} 

Let us define $\psi(u):=|u|^{p-2}u+|u|^{q-2}u$ and a map $B_\lambda\colon \W \to \W$ by
\begin{equation*}
	B_\lambda(u):=T_\lambda^{-1}(h(\cdot,u) + \lambda \psi(u)) 
\end{equation*}
for $u\in \W$ and $\lambda>0$. 
According to Lemma~\ref{T} and the assumption \eqref{h}, we see that 
$B_\lambda$ is well-defined and continuous. 
Moreover, critical points of the energy function $J$ given by \eqref{eq:J} correspond to fixed points of $B_\lambda$, see \cite[Remark~12]{MT2}. 
Throughout this section, $K:=\{u \in  \W:~J'(u) = 0\}$ is the set of critical points of $J$, and, to shorten notation, we write $\|u\|$ instead of $\|\nabla u\|_p$ for $u \in \W$. 

By the standard calculations, we have the following facts (cf. \cite[Lemmas~3.7 and 3.8]{bartschliu} for details). 

\begin{lemma}\label{B:estimate} 
	Let $\lambda>0$. 
	Then there exist constants  $d_i=d_i(\lambda)>0$,  $i=1,2,\dots,6$, such that for all $u\in \W$ the following assertions hold:
	\begin{itemize} 
		\item[{\rm (i)}] 
		$\left< J'(u), u-B_\lambda(u) \right> \ge 
		d_1\|u-B_\lambda(u)\|^2\left((\|u\|+\|B_\lambda(u)\|)^{p-2}+(\|u\|+\|B_\lambda(u)\|)^{q-2}\right)$  for $1 < q < p \leq 2$; 
		\item[{\rm (ii)}] 
		$\left< J'(u), u-B_\lambda(u) \right> \ge 
		d_2\left(\|u-B_\lambda(u)\|^p+\|u-B_\lambda(u)\|^q\right)$ 
		for $2 \leq q < p$; 
		\item[{\rm (iii)}] $\left< J'(u), u-B_\lambda(u) \right> \ge 
		d_3\|u-B_\lambda(u)\|^2(\|u\|+\|B_\lambda(u)\|)^{q-2}
		+d_3\|u-B_\lambda(u)\|^p$ 
		for $1 < q \leq 2 \leq p$; 
		\item[{\rm (iv)}] 
		$\|J'(u)\|_{(\W)^*}\le 
		d_4\left(\|u-B_\lambda(u)\|^{p-1}+\|u-B_\lambda(u)\|^{q-1}\right)$
		for $1<q<p\le 2$; 
		\item[{\rm (v)}] 
		$\|J'(u)\|_{(\W)^*}\le d_5\|u-B_\lambda(u)\|
		\left((\|u\|+\|B_\lambda(u)\|)^{p-2}+(\|u\|+\|B_\lambda(u)\|)^{q-2}\right)$ 
		for $2 \leq q < p$;
		\item[{\rm (vi)}] $\|J'(u)\|_{(\W)^*}\le d_6\|u-B_\lambda(u)\|
		(\|u\|+\|B_\lambda(u)\|)^{p-2}+d_6\|u-B_\lambda(u)\|^{q-1}$ 
		for $1 < q \leq 2 \le p$. 
	\end{itemize}
\end{lemma}

Then, similar arguments as in \cite[Lemma~17]{MT2} (see also \cite[Lemma 4.1]{bartschliu}) can be applied to prove the following result on the existence of a locally Lipschitz continuous pseudo-gradient vector field in order to produce an invariant descending flow with respect to the positive and negative cones $\pm P$ defined by \eqref{eq:P}.
\begin{lemma}\label{pgvf} 
	Let $\lambda > \lambda_0$, where $\lambda_0 > 0$ is given by the assumption $(A1)$ of Theorem~\ref{thm:1}. 
	Then, there exists a locally Lipschitz continuous operator $V_\lambda\colon \C\setminus K \to \C$ such that the following hold:
	\begin{itemize} 
		\item[{\rm (i)}]
		For any $u \in C^1_0(\overline{\Omega})\setminus K$ we have
		\begin{itemize} 
			\item[]\hspace*{-8mm}  
			$\left< J'(u), u-V_\lambda(u) \right> \ge 
			\frac{d_1}{2}\|u-B_\lambda(u)\|^2\left\{(\|u\|+\|B_\lambda(u)\|)^{p-2}+(\|u\|+\|B_\lambda(u)\|)^{q-2}\right\}$ 
			for $1 < q < p \leq 2$; 
			\item[]\hspace*{-8mm} 
			$\left< J'(u), u-V_\lambda(u) \right> \ge 
			\frac{d_2}{2}\left(\|u-B_\lambda(u)\|^p+\|u-B_\lambda(u)\|^q\right)$ 
			for $2 \leq q < p$; 
			\item[]\hspace*{-8mm} 
			$\left< J'(u), u-V_\lambda(u) \right> \ge 
			\dfrac{d_3}{2}\|u-B_\lambda(u)\|^2(\|u\|+\|B_\lambda(u)\|)^{q-2}
			+\dfrac{d_3}{2}\|u-B_\lambda(u)\|^p$ 
			for $1 < q \leq 2 \leq p$; 
			\item[]\hspace*{-8mm} 
			$\frac{1}{2}\|u-B_\lambda(u)\|\le \|u-V_\lambda(u)\| \le 
			2\|u-B_\lambda(u)\|$. 
		\end{itemize} 
		Here $d_1$, $d_2$, and $d_3$ are the positive constants from Lemma~\ref{B:estimate}.
		\item[{\rm (ii)}] $V_\lambda(u)\in \pm\, {\rm int}\,P$ 
		for every $u \in \pm\,P \setminus K$, 
		respectively.
		\item[{\rm (iii)}] 
		Let $p^*:= \frac{Np}{N-p}$ for $N>p$, and $p^* := p+1$ otherwise. 		
		Set $r_0:=p^*$ and define a sequence $\{r_n\}_{n \in \mathbb{N}}$ inductively as follows: 
		$$
		r_{n+1}:=p^*r_n/p=(p^*/p)^{n+1}p^*.
		$$
		Then, for any $n\in\mathbb{N}$ 
		there exists a constant $C_n^* > 0$ such that 
		$$
		\|V_\lambda(u)\|_{r_{n+1}} \leq C_{n+1}^* (2+|\Omega|+\|u\|_{r_n}) 
		\quad \text{for all} \quad 
		u \in \C \setminus K; 
		$$
		\item[{\rm (iv)}] If $N \geq p$ and $r>\max\{N/p,1/(p-1)\}$, then 
		there exists a constant $D_1 > 0$ such that 
		$$
		\|V_\lambda(u)\|_\infty \leq D_1(\|u\|_{r(p-1)}+2+|\Omega|) \quad \text{for all} \quad
		u \in \C \setminus K; 
		$$
		\item[{\rm (v)}] There exists a constant $D_2 > 0$ such that 
		$$
		\|V_\lambda(u)\|_\infty \leq D_2(2+\|u\|_\infty) 
		\quad \text{for all} \quad 
		u\in \C\setminus K; 
		$$
		\item[{\rm (vi)}] For every $R>0$ there exist $\gamma\in(0,1)$ 
		and $M>0$ such that 
		$\|V_\lambda(u)\|_{C^{1,\gamma}_0(\overline{\Omega})} \leq M$ 
		for all $u \in \C \setminus K$ 
		with $\|u\|_\infty \leq R$. 
	\end{itemize} 
\end{lemma}

Now, we will give the proof of Theorem~\ref{thm:1}. 

\begin{proof*}{Theorem~\ref{thm:1}} 
	Note first that the boundary of $\pm P$ in $\C$ does not intersect with $K \setminus \{0\}$, since any nonnegative (resp. nonpositive) 
	and nontrivial solution of corresponding equation is strictly positive (resp. negative) in $\Omega$ and $\partial u/\partial \nu<0$ (resp. $>0$) on $\partial\Omega$ under the assumption $(A1)$ of the theorem, due to the strong maximum principle and boundary point lemma (see \cite[Theorem~5.3.1 and Theorem~5.5.1]{PS}). 
	
	Take $\lambda > \lambda_0$ and let $V_\lambda$ be a locally Lipschitz continuous operator given by Lemma~\ref{pgvf}. 
	Consider the following initial value problem in 
	$C^{1}_0(\overline{\Omega})$: 
	$$
	\left\{ 
	\begin{aligned}
	\dfrac{d\eta}{dt}(t) &= -\eta(t)+V_\lambda(\eta(t)),\\
	\eta(0) &= u.
	\end{aligned}
	\right. 
	$$
	Denote by $\eta(t,u) \in C^{1}_0(\overline{\Omega})$ its unique solution on the right maximal interval $[0,\tau(u))$.
	According to the assertion (ii) of Lemma~\ref{pgvf}, $\eta(t,u)$ is the  \textit{invariant descending flow} with respect to the positive cone $P$ and the negative cone $-P$, 
	namely, $\eta(t,u)\in \pm {\rm int}\, P$ for all $0<t<\tau(u)$ 
	provided $u\in \pm P\setminus K$ (see \cite[Lemma~3.2]{LS}). 
	Define the sets 
	\begin{equation*}
		Q_\pm:=\{u \in \C \setminus K:~ 
		\eta(t,u)\in \pm\,{\rm int}\,P
		\ \text{ for some }\ t\in[0,\tau(u))
		\,\}\cup (\pm\,{\rm int}\,P). 
	\end{equation*}
	It is known that $Q_\pm$ are open subsets of $\C$ invariant for the descending flow $\eta$, and $\partial Q_\pm$ are closed subsets of $\C$ invariant for $\eta$, see \cite[Lemma~2.3]{LS}. 
	
	Choose a constant $c$ satisfying 
	$\max\limits_{s \in [0,1]} J(\gamma(s))<c<0$, 
	where $\gamma$ is the continuous path given by the assumption $(A2)$ of the theorem. 
	Since $\gamma(0)\in Q_+$, $\gamma(1)\in Q_-$, and $Q_\pm$ are open in $\C$, 
	there exist  $0<s_+ \leq s_-<1$ such that 
	$\gamma(s_+)\in \partial Q_+$ and $\gamma(s_-)\in \partial Q_-$. 
	Put $u_1:=\gamma(0)$, $u_2:=\gamma(1)$, and $u_3:=\gamma(s_+)$. 
	Due to the assertion (i) of Lemma~\ref{pgvf}, we know that
	$$
	\frac{d}{dt} J(\eta(t,u_i))=-\left< J'(\eta(t,u_i)), 
	\eta(t,u_i)-V_\lambda(\eta(t,u_i))\right> \leq 0, \quad i=1,2,3,
	$$ 
	which implies that 	$-\infty< \inf_{\W} J \leq J(\eta(t,u_i)) \leq c<0$ for every $t\in[0,\tau(u_i))$. Hence, the coercivity of $J$ 
	guarantees the existence of $R>0$ such that for all $t\in[0,\tau(u_i))$ we have
	\begin{equation}\label{thm-3-1}
		\|\eta(t,u_i)\| \leq R
		\quad  \text{and} \quad 
		\|B_\lambda(\eta(t,u_i))\| \leq R.
	\end{equation} 
	Therefore, if $\tau(u_i)<\infty$ for $i=1,2,3$, then for every $0<t_1<t_2<\tau(u_i)<\infty$ we have 
	\begin{align*}
		\|\eta(t_1,u_i)-\eta(t_2,u_i)\| &\le 
		\int_{t_1}^{t_2}\|\eta(s,u_i)-V_\lambda(\eta(s,u_i))\|\,ds \\
		& \le 2 \int_{t_1}^{t_2}\|\eta(s,u_i)-B_\lambda(\eta(s,u_i))\|\,ds 
		\le 4R (t_2-t_1)
	\end{align*}
	by the assertion (i) of Lemma~\ref{pgvf} and \eqref{thm-3-1}. 
	Thus, $\eta(t,u_i)$ converges to some $w_i$ in $\W$ 
	as $t \to \tau(u_i)-0$ whenever $\tau(u_i)<\infty$. 
	On account of Lemma~\ref{pgvf} and \cite[Lemma~18 (ii)]{MT2}, it is not hard to prove that $w_i\in K$ and 
	$\eta(t,u_i)$ converges to $w_i$ in $\C$ as $t \to \tau(u_i)-0$. 
	Recalling now that $Q_\pm$ and $\partial Q_\pm$ are invariant, we see that $J(w_i) \leq J(u_i) \leq c<0$, $i=1,2,3$, and $w_1 \in {\rm int}\,P$, $w_2 \in -{\rm int}\,P$,  $w_3 \in \partial Q_+$. 
	Since $\partial Q_+\cap (\pm P\setminus\{0\})=\emptyset$ 
	(note that $\pm P\setminus\{0\} \subset Q_\pm$), our conclusion is proved provided $\tau(u_i)<\infty$ for $i=1,2,3$. 
	
	Assume that $\tau(u_i)=\infty$ for some $i \in \{1,2,3\}$. 
	In this case, we can prove the existence of a sequence $\{t_n\}_{n \in \mathbb{N}} \subset \mathbb{R}^+$ such that 
	\begin{equation}\label{claim}
		t_n \to +\infty 
		\quad \text{and} \quad J'(\eta(t_n,u_i)) \to 0
		\quad  \text{in} \quad (\W)^* 
		\quad \text{as} \quad 
		n \to +\infty. 
	\end{equation}
	Note that 
	there exists a sequence $\{t_n\}_{n \in \mathbb{N}} \subset \mathbb{R}^+$ 
	such that $t_n \to +\infty$ and $\frac{d}{dt} J(\eta(t_n,u_i)) \to 0$ 
	as $n \to +\infty$, since $-\infty < \inf_{\W} J \leq J(\eta(t,u_i)) \le c$ for all $t \geq 0$ and $J(\eta(t,u_i))$ is nondecreasing in $t$. 
	Let us show that this sequence satisfies \eqref{claim}. If $1 < q < p \leq 2$, then Lemma~\ref{B:estimate} (iv),  Lemma~\ref{pgvf} (i), and \eqref{thm-3-1} imply
	\begin{align*} 
		-\frac{d}{dt} J(\eta(t,u_i)) 
		& \ge \frac{d_1}{2}\dfrac{\|\eta(t,u_i)-B_\lambda(\eta(t,u_i))\|^2}
		{(\|\eta(t,u_i)\|+\|B_\lambda(\eta(t,u_i))\|)^{2-p}
			+(\|\eta(t,u_i)\|+\|B_\lambda(\eta(t,u_i))\|)^{2-q}}
		\\[2mm]
		&\ge \frac{d_1}{2}\dfrac{\|\eta(t,u_i)-B_\lambda(\eta(t,u_i))\|^2}
		{(2R)^{2-p}+(2R)^{2-q}}
		\\[2mm] 
		&\ge \dfrac{d_1}
		{2d_4^{2/(q-1)}(1+(2R)^{p-q})^{2/(q-1)}\{(2R)^{2-p}+(2R)^{2-q}\}}
		\|J'(\eta(t,u_i))\|_{(\W)^*}^{2/(q-1)}
	\end{align*} 
	for all $t>0$. Hence $\|J'(\eta(t_n,u_i))\|_{(\W)^*} \to 0$ 
	as $n \to +\infty$. 
	The cases $2 \leq q < p$ and $1 < q \leq 2 \leq p$ can be handled in a similar way using the estimates of Lemma~\ref{B:estimate} and  Lemma~\ref{pgvf} (i). 
	
	Combining now \eqref{claim} with \eqref{thm-3-1}, we conclude that 
	$\{\eta(t_n,u_i)\}_{n \in \mathbb{N}}$ is a bounded Palais--Smale sequence to $J$. 
	At the same time, it is not hard to show that $J$ satisfies the Palais--Smale condition because the coercivity of $J$ implies the boundedness of any Palais--Smale sequence (see Lemma~\ref{PS}). 	
	Thus, there exists $w_i\in \W\cap K$ such that 
	$\lim\limits_{n \to +\infty}\eta(t_n,u_i)=w_i$ in $\W$, up to an appropriate subsequence. 
	Furthermore, arguing as in the proof of \cite[Lemma~18 (iii)]{MT2}, using Lemma~\ref{pgvf} (iii)-(vi) 
	and \eqref{thm-3-1}, 
	we see that $\{\eta(t,u_i):~ t \geq 0\}$ is bounded in 
	$C^{1,\nu}_0(\overline{\Omega})$ for some $\nu \in (0,1)$. 
	Thus, the compactness of 
	$C^{1,\nu}_0(\overline{\Omega})\hookrightarrow \C$ and 
	$\lim\limits_{n \to +\infty}\eta(t_n,u_i)=w_i$ in $\W$ imply that $\lim\limits_{n \to +\infty}\eta(t_n,u_i)=w_i$ in $\C$.
	Therefore, $w_1 \in {\rm int}\,P$, 	$w_2 \in -{\rm int}\,P$ and $w_3 \in \C\setminus (P\cup -P)$. 
\end{proof*}


\end{document}